\RequirePackage{fix-cm}
\RequirePackage{amsmath}
\documentclass[final,1p,times]{elsarticle}
\usepackage[utf8]{inputenc}

\usepackage{url} 
\usepackage{multirow} 

\usepackage{xfrac}
\usepackage[shortlabels]{enumitem}
\setlist{nosep}
\newlist{algolist}{enumerate}{3}
\setlist[algolist,1]{label=(\arabic*), ref=(\arabic*)}
\setlist[algolist,2]{label=(\alph*), ref=(\arabic{algolisti})-(\alph*)}
\setlist[algolist,3]
{label=(\roman*), ref=(\arabic{algolisti})-(\alph{algolistii})-(\roman*)}
\usepackage{color}
\usepackage{datetime}
\usepackage{soul}
\usepackage{verbatim}

\newcommand{\BP}{Barrett and Prigozhin}
\newcommand{\RU}{R\"uschendorf and Uckelmann}
\newcommand{\Was}{Wasserstein}

\usepackage{graphicx}
\usepackage[percent]{overpic}
\usepackage[caption=false]{subfig}
\usepackage{float}
\newfloat{algorithm}{t}{lop}
\usepackage{textpos} 
\usepackage{amssymb}              
\usepackage{amsfonts}
\usepackage{bm}                   
\usepackage{bbm}                  
\usepackage{array}                
\usepackage{mathtools}            
\numberwithin{equation}{section}  
\allowdisplaybreaks

\newtheorem{theorem}{Theorem}[section]
\newtheorem{lemma}[theorem]{Lemma}
\newtheorem{corollary}[theorem]{Corollary}
\newtheorem{definition}[theorem]{Definition}
\newtheorem{example}[theorem]{Example}
\newtheorem{rem}{Remark}
\newproof{proof}{Proof}

\DeclarePairedDelimiter\abs{\lvert}{\rvert} 
\DeclareMathOperator*{\arginf}{arg\,inf} 
\DeclareMathOperator*{\argsup}{arg\,sup} 
\providecommand{\CC}{{{C\nolinebreak[4]\hspace{-.05em}\raisebox{.4ex}
{\tiny\bf ++}}}}
\providecommand{\cl}[1]{\bar{#1}}
\providecommand{\clBrp}{\cl{B}^r_{\scriptscriptstyle{+}}}
\providecommand{\Conv}{\mathrm{Conv}}
\providecommand{\edge}{\mathrm{edg}} 
\providecommand{\inter}{\mathrm{int}} 
\providecommand{\N}{\mathbb{N}} 
\DeclarePairedDelimiter\norm{\lVert}{\rVert} 
\providecommand{\R}{\mathbb{R}} 
\providecommand{\Z}{\mathbb{Z}} 
\DeclarePairedDelimiter\set{\{}{\}}
\providecommand{\vc}[1]{\mathbf{#1}}

\providecommand{\aref}[1]{Step~\ref{#1}}        
\providecommand{\cref}[1]{Condition~\ref{#1}}   
\providecommand{\dref}[1]{Definition~\ref{#1}}  
\providecommand{\eref}[1]{Equation~\eqref{#1}}  
\providecommand{\fref}[1]{Figure~\ref{#1}}      
\providecommand{\lref}[1]{Lemma~\ref{#1}}       
\providecommand{\rref}[1]{Remark~\ref{#1}}      
\providecommand{\sref}[1]{Section~\ref{#1}}     
\providecommand{\tref}[1]{Theorem~\ref{#1}}     
\providecommand{\xref}[1]{Example~\ref{#1}}     
\providecommand{\fsubref}[2]{Figure~\ref{#1}\protect\subref{#2}}   
\begin{document}
\begin{frontmatter}
\title{The boundary method for semi-discrete optimal transport partitions and
Wasserstein distance computation\tnoteref{NSF}}
\tnotetext[NSF]{This material is based upon work
supported by the National Science Foundation Graduate Research Fellowship 
Program under Grant No. DGE-1650044. Any opinions, findings, and conclusions or 
recommendations expressed in this material are those of the authors and do not 
necessarily reflect the views of the National Science Foundation.}
\author{Luca Dieci}
\address{School of Mathematics,
              Georgia Institute of Technology, 
              Atlanta, GA 30332 U.S.A.\\
              Tel.: +1 404-894-9209
              Fax: +1 404-894-4409}
\ead{dieci@math.gatech.edu}
\author{J.D.\ Walsh III}
\address{Naval Surface Warfare Center,
              Panama City Division (X24),
              110 Vernon Ave.,
              Panama City, FL 32407 U.S.A.\\
              Tel.: +1 850-234-4660
              Fax: +1 850-235-5374}
\ead{joseph.d.walsh@navy.mil}

\begin{abstract}
We introduce a new technique, which we call the {\emph{boundary method}},
for solving semi-discrete optimal transport problems with a wide range of cost
functions.
The boundary method reduces the effective dimension of the problem, thus improving
complexity.
For cost functions equal to a $p$-norm with $p \in (1,\infty)$,
we provide mathematical justification, convergence analysis, and algorithmic
development. Our testing supports the boundary method with these $p$-norms,
as well as other, more general cost functions.
\end{abstract}
\end{frontmatter}

\section{Introduction}
In this work, we consider a new solution method for optimal transport problems.
Numerical optimal transport has applications in a wide range of fields, but 
the scaling properties and ground cost restrictions of current numerical
methods make it difficult to find solutions for many applications.

The boundary method we propose focuses on a broad class of
optimal transportation problems: {\emph{semi-discrete optimal transport}}.
Many other techniques assume semi-discrete transport, either implicitly
or explicitly, as
semi-discrete formulations can be used to approximate solutions to fully 
continuous problems, and the semi-discrete optimal transport problem is of 
practical relevance itself.

Key challenges in numerical optimal transport are: 
(a) the design of numerical methods capable of handling general ground costs,
(b) efficient computation of the \Was{} metric, and
(c) solutions of three (or higher) dimensional problems.
The boundary method addresses these
concerns by solving problems where the ground cost is a $p$-norm,
$p \in (1,\,\infty)$, and by doing so in a way that reduces the effective
dimension of the transport problem.

\subsection{Description of optimal transport: the
Monge-Kantorovich problem}\label{s:transport_problem}
The theory of optimal transport dates back to the work by Monge in 1781, 
~\cite{Monge1781a}.
In the 1940s, Kantorovich's papers~\cite{Kantorovich1942a,Kantorovich1948a} 
relaxed Monge's requirement that no mass be split, creating we now know 
as the Monge-Kantorovich problem.

\begin{definition}[Monge-Kantorovich problem]\label{MKproblem}
Let $X,\,Y \subseteq \R^d$, let $\mu$ and $\nu$ be probability densities 
defined on $X$ and $Y$, and let $c(\vc{x},\,\vc{y}) : X \times Y \to \R$ be a 
measurable 
\emph{ground cost} function.
Define the set of \emph{transport plans}
\begin{equation}\label{MK-1}
\Pi(\mu,\,\nu) := \set*{\pi \in \mathcal{P}(X \times Y) \left|
\begin{array}{c}
\pi[A \times Y] = \mu[A],\,
 \pi[X \times B] = \nu[B] \ ,\\
 \forall \text{ meas.\ }
 A \subseteq X,\, B \subseteq Y
\end{array}
\right. },
\end{equation}
where $\mathcal{P}(X \times Y)$ is the set of probability measures on the 
product space,
and define the \emph{primal cost} function $P: \Pi(\mu,\,\nu) \to \R$ as
\begin{equation}\label{primal_cost}
P(\pi) := \int_{X \times Y} c(\vc{x},\,\vc{y})\, d\pi(\vc{x},\,\vc{y}).
\end{equation}
The Monge-Kantorovich problem is to
find the \emph{optimal primal cost}
\begin{equation}\label{MK-2}
P^* := \inf_{\pi \in \Pi(\mu,\,\nu)}\, P(\pi),
\end{equation}
and an associated \emph{optimal transport plan}
\begin{equation}\label{MK-pi-star}
\pi^* := \arginf_{\pi \in \Pi(\mu,\,\nu)}\, P(\pi).
\end{equation}
\end{definition}

\noindent
Kantorovich also identified the problem's dual formulation.
\begin{definition}[Dual formulation]\label{KantoDual}
Define the set of functions
\begin{equation}
\Phi_c(\mu,\,\nu) := \set*{(\varphi,\,\psi) \in L^1(d\mu) \times L^1(d\nu)
\left|
\begin{array}{c}
\varphi(\vc{x}) + \psi(\vc{y}) \leq c(\vc{x},\,\vc{y}) \ ,\\
d\mu \text{ a.e.\ }
\vc{x} \in X,\,d\nu\text{ a.e.\ } \vc{y} \in Y
\end{array}
\right. }.
\end{equation}
Let the \emph{dual cost} function,
$D: \Phi_c(\mu,\,\nu) \to \R$, be defined as
\begin{equation}\label{e:dualCost}
D(\varphi,\,\psi) := \int_X \varphi \,d\mu + \int_Y \psi \,d\nu.
\end{equation}
Then, the \emph{optimal dual cost} is
\begin{equation}
D^* := \sup_{(\varphi,\,\psi) \in \Phi_c(\mu,\,\nu)}\, D(\varphi,\,\psi),
\end{equation}
and an optimal dual pair is given by
\begin{equation}
(\varphi^{*},\,\psi^{*}) := \argsup_{(\varphi,\,\psi) \in \Phi_c(\mu,\,\nu)}\, 
D(\varphi,\,\psi).
\end{equation}
\end{definition}

When the ground cost is a distance function (often but not
necessarily Euclidean), Monge-Kantorovich solutions are related to the
\emph{\Was{} metric}, a distance between probability distributions:
\begin{equation}\label{WassMetric}
W_1(\mu,\,\nu) := \inf_{\pi \in \Pi(\mu,\,\nu)}\,
\int_{X \times Y} c(\vc{x},\,\vc{y}) \,
d\pi(\vc{x},\,\vc{y}).
\end{equation}
We have $W_1(\mu,\,\nu) = P^* = D^*$, and hence, we may refer to any of these 
as the \emph{Wasserstein distance}, the \emph{optimal transport cost}, or simply 
the \emph{optimal cost}.\footnote{See also~\cite[p.\ 207]{Villani2003a},
a definition of the \Was{} metric $W_p$ with $p \in [0,\,\infty)$.}

\begin{rem}
$W_1(\mu,\,\nu)$ is often written as $W_1$, with $\mu$ and $\nu$ implied.
Furthermore, as \eref{WassMetric} makes clear, $W_1(\mu,\,\nu)$ also depends 
on the ground cost function $c(\vc{x},\,\vc{y})$.
In the literature, the \Was{} distance formula often assumes
the ground cost to be a specific predetermined function, usually the
Euclidean distance $\norm{\vc{x}-\vc{y}}_2$.
\end{rem}

\begin{definition}[Monge problem]\label{Mproblem}
In certain cases, there exists at least one solution to the semi-discrete 
Monge-Kantorovich problem that does not split transported masses.
In other words, there exists some $\pi^*$ such that
\begin{equation}\label{Tstar}
\pi^*(\vc{x},\,\vc{y}) = \pi^*_{\scriptscriptstyle{T^*}}(\vc{x},\,\vc{y}) :=
\mu(\vc{x})\,\delta[\vc{y} = T^*(\vc{x})],
\end{equation}
where $T^* : X \to Y$ is a measurable map called \emph{optimal
transport map}.\footnote{One can also write $\pi^*_{T^*}$ as $(\mathrm{Id}
\times T^*) \# \mu$. Our notation is from~\cite[p.\ 3]{Villani2003a}. The alternative notation is
used in \cite{Gangbo1996a}.}
When such a $\pi^*$ exists, we say the solution also solves the Monge 
problem.

If the Monge-Kantorovich problem has a solution which solves the Monge
problem, we can assume without loss of 
generality that every $\pi \in \Pi(\mu,\,\nu)$ satisfies
\begin{equation}\label{Tplan}
\pi(\vc{x},\,\vc{y}) = \pi_{\scriptscriptstyle{T}}(\vc{x},\,\vc{y}) :=
\mu(\vc{x})\,\delta[\vc{y} = T(\vc{x})],
\end{equation}
for some measurable transport map $T : X \to Y$, and that the primal cost can 
be written
\begin{equation}\label{Mprimal_cost}
P(\pi) := \int_{X} c(\vc{x},\,T(\vc{x}))\, d\mu(\vc{x}).
\end{equation}
\end{definition}

\subsection{Semi-discrete problem}\label{s:semi-discrete}
The semi-discrete optimal transport problem we consider is the Monge-Kantorovich
problem of \dref{MKproblem}, with restrictions on $\mu$ and $\nu$, and $c$.
\begin{center}
\fbox{
  \parbox[c]{4.5in}{
\begin{enumerate}[label=(\arabic*)]
\item
Assume that $\mu$ satisfies the following:
\begin{enumerate}[label=(\alph*),ref=\theenumi(\alph*)]
\item\label{c:absCont}
$\mu$ is absolutely continuous with respect to the Lebesgue measure.
\item\label{c:convComp}
The support of $\mu$ is contained in the convex compact region $A \subseteq X$.
\\
(Since $A \subset \R^d$, it must also be the case that $A$ is simply connected.)
\end{enumerate}
\item\label{c:nuFinite}
Assume $\nu$ has exactly $n \geq 2$ non-zero values, located at 
$\set{\vc{y}_i}_{i=1}^n 
\subseteq Y$.
\item\label{c:pNorm}
Assume $c$ is a $p$-norm with $p \in (1,\,\infty)$.
\end{enumerate}
  }
}
\end{center}
As we will show, each of these conditions is required for one or more of the
theorems given in \sref{s:math}.
\cref{c:absCont} ensures that the value of $\mu$ is bounded, which is required to show Wasserstein distance convergence in \tref{t:WassErr}.
Conditions \ref{c:absCont}, \ref{c:convComp}, \ref{c:nuFinite}, and
\ref{c:pNorm} are all used to satisfy the conditions of
Corollary~4\footnote{See \tref{t:CuestaAlbertos}, below, for a full statement of this result.}
of \cite{Cuesta1993a}, which we apply to show the $\mu$-a.e.\ uniqueness of the solution in \tref{t:unique_solution}.

\subsubsection{Semi-discrete transport 
and the Monge problem}\label{s:semidcosts}
Since $\mu$ is absolutely continuous, $\abs{S}=0$ implies $\mu(S)=0$ for all Borel sets
$S$ in $X$. Hence, $\mu$ is nonatomic.
Because $c$ is continuous and $\mu$ is nonatomic, at least 
one solution to the semi-discrete Monge-Kantorovich problem also satisfies 
the Monge problem, described in \dref{Mproblem}; see Theorem B in~\cite{Pratelli2007a}.
Thus, by applying \eref{Tplan}, we can assume without loss of generality that 
any transport plan $\pi$ partitions $A$ into $n$ sets $A_i$, 
where $A_i$ is the set of points in $A$ that are transported by the map $T$ to 
$\vc{y}_i$.
Using this partitioning scheme in combination with \eref{Mprimal_cost} allows 
us to rewrite the primal cost function 
for the semi-discrete problem as
\begin{equation}\label{primal_cost_sd}
P(\pi) := \sum_{i=1}^n \int_{A_i} c(\vc{x},\,\vc{y}_i)\, 
d\mu(\vc{x}).
\end{equation}

\subsection{Shift characterization for semi-discrete 
optimal transport}\label{s:oc_form}
Using this idea of sets $A_i$, we are ready to describe the shift 
characterization of the semi-discrete optimal transport problem.
The definition of the characterization, which follows, is based on one given by 
R\"{u}schendorf and Uckelmann in~\cite{Ruschendorf2007a,Ruschendorf2000a}.

\begin{definition}[Shift characterization]\label{ShiftChar}
Let $\set{a_i}_{i=1}^n$ be a set of $n$ finite values, referred 
to as \emph{shifts}.
Define
\begin{equation}\label{Fdef}
F(\vc{x}) := \max_{1\leq i\leq n} \set{a_i - c(\vc{x},\,\vc{y}_i)}.
\end{equation}
For $i \in \N_n$, where $\N_n = \set{1,\,\ldots,\,n}$, let
\begin{equation}\label{Aidef}
A_i  := \set{\vc{x} \in A \mid F(\vc{x}) = a_i - c(\vc{x},\,\vc{y}_i)}.
\end{equation}
Note that $\cup_{i=1}^nA_i=A$.
The problem of determining an optimal transport plan $\pi^*$ is equivalent to 
determining shifts $\set{a_i}_{i=1}^n$ such that for all $i \in 
\N_n$, the total mass transported from $A_i$ to $\vc{y}_i$ equals 
$\nu(\vc{y}_i)$.
\end{definition}

The shift characterization is derived from the dual cost function given in
\eref{e:dualCost}.
For any $D(\varphi,\,\psi)$,
suppose we define
\begin{equation}
\varphi'(\vc{x}) = \sup_{\vc{y} \in Y} \set{\psi(\vc{y}) - c(\vc{x},\vc{y})}. 
\end{equation}
Then $D(\varphi',\,\psi) \geq D(\varphi,\,\psi)$ for all $\psi$.

For the semidiscrete problem, $\varphi'$ is exactly \eref{Fdef}, and the shifts
$a_i$ correspond to the value of $\psi$ at each Dirac mass $\vc{y}_i$.
Hence, the discrete problem is no more than a special case of the general
continuous problem where $\mu$  is a continuous density function and
$\nu$ an empirical measure. For a detailed derivation, see~\cite{Gangbo1996a}.

In the same way, the sets $A_i$ correspond to the subdifferentials
$\partial_c(\vc{y}_i)$. For a general cost function $c$, the sets $A_i$ are referred to
in analysis as \emph{Laguerre cells}, and the map generated by the sets $A_i$ over $A$
is called a \emph{Laguerre diagram}.
As we will discuss further on, the boundaries between Laguerre cells are typically sections of hypersurfaces.
When $c(\vc{x},\,\vc{y}) = \norm{\vc{y}-\vc{x}}_2^2$, the boundaries are sections of
hyperplanes, and the map as called a \emph{power diagram}. See \cite{Aurenhammer1987a}
for a detailed evaluation of this special case.
There are also cost functions where, for certain arrangements of $\set{\vc{y}_i}$, the boundaries
between Laguerre cells have positive Lebesgue measure in $\R^d$.
An example is shown in \fsubref{f:exact}{f:badManhattan}.

\subsection{Numerical approaches to the MK 
problem}\label{s:numerical_mk}
Applications of optimal transport are found in many areas of research, 
including medicine, economics, image processing, machine learning, physics, and
many others; e.g.,
see~\cite{Muskulus2010b,Carlier2001a,Haker2001a,Cuturi2013a,Carlen2004a}.
For that reason, many people have focused their research on 
numerical methods for the Monge-Kantorovich problem.

The solution to a semi-discrete problem can be approximated by treating the
problem as fully discrete, and the solution to a fully continuous problem can be
approximated by treating it as either semi- or fully discrete.
By ``treating,'' we refer primarily to assumptions about continuity:
in practice, nearly every approach fully discretizes the problem, and the 
complexity of such approaches is relative to the measure of the discretization.

The semi-discrete problem has received significant attention in its role as
a discretization of the continuous problem (where continuity assumptions are
employed over $X$ but not $Y$). Substantial effort has been taken to quantify
the extent to which solutions to such semi-discrete problems approximate
the solution to the original continuous problem; for example,
see~\cite{Merigot2011a}.
However, the semi-discrete problem has interesting applications in its own right. 
Recent developments include works in economics~\cite{Chiappori2010a,Chiappori2016a,Dupuis2016a},
image processing~\cite{deGoes2012a}, and
optics~\cite{Abedin2016a,Glimm2003a}.
In addition, the power and flexibility of Laguerre cell tesselation (vs.\ Voronoi) drive
ongoing research in physics and other fields.

When the ground cost for the semi-discrete problem is the squared $2$-norm,
$\norm{\cdot}_2^2$,
significant numerical progress has been achieved.
In 1988, Oliker and Prussner introduced what came to be called the Oliker-Prussner
algorithm for nonlinear Monge-Amp\`{e}re-type equations in $\R^2$; see~\cite{Oliker1988a}.
Oliker and Prussner were significantly ahead of their time.
A 1992 paper by Aurenhammer et al.,~\cite{Aurenhammer1992a}, while describing
a different algorithm (Newton's method), explicitly connected the
Oliker and Prussner's approach to
semi-discrete transport and its resulting ``Voronoi-type diagrams.''
In 1998 Aurenhammer et al.\ published~\cite{Aurenhammer1998a}, a revision that
clarified important details, and incorporated an argument from~\cite{Cuesta1993a}
to guarantee that the sets $A_i$ partition $A$ $\mu$-a.e.
More recent algorithms appear in~\cite{Kitagawa2017a,Merigot2011a}.

When sets $A_i$ and $A_j$ share a boundary, for some $i \neq j$, there is a
monotone relationship between the volume of $A_i$ and the
difference of shifts, $a_i-a_j$. The Oliker-Prussner approach and the boundary
method both exploit this relationship, though in very different ways.
Whether applying the Oliker-Prussner algorithm or some variation such as Newton's method,
the Oliker-Prussner approach begins with approximated sets $\tilde{A}_i$, and
directly perturbs the approximated shift difference $\tilde{a}_i-\tilde{a}_j$
in order to bring $\mu(\tilde{A}_i)$ closer to $\nu(\vc{y}_i)$.
This approach is extended over all the shift differences,\footnote{They refer to
a set of shift differences $\set{a_i-a_j \mid i,\,j \in \N_n,\,i < j}$ as a \emph{weight vector}.}
making it, in essence, a method for solving the Monge-Kantorovich dual problem
with $c = \norm{\cdot}_2^2$.
Because the squared $2$-norm is strictly convex, and it ensures that the boundary
for each adjacent $A_i$ and $A_j$ is a hyperplane, algorithms based on the
Oliker-Prussner approach are generally able to quantify convergence behavior
and guarantee termination after a finite number of refinement steps.

Numerous efforts have been made to extend the approach proposed by Oliker
and Prussner.
An application-focused paper by Caffarelli et al.\ extends the Oliker-Prussner 
algorithm to $\R^3$, assuming special geometries~\cite{Caffarelli1999a}.
L\'{e}vy presents a parallelized Newton's method for three
dimensions, one which scales well when $Y$ consists of large numbers of Dirac
masses~\cite{Levy2015a}.
Other works, such as ~\cite{Mirebeau2015a}, attempt to integrate the
Oliker-Prussner approach with the Wide Stencil methods developed for continuous
Monge-Amp\`{e}re problems; see, e.g.,~\cite{Benamou2014a,Oberman2006a}.
All of these assume $c = \norm{\cdot}_2^2$.

A few authors have attempted to develop approaches
for ground costs other than the squared $2$-norm. Most of these
do not employ Oliker-Prussner.
In~\cite{Ruschendorf2000a}, \RU{} report on numerical experiments with
ground costs given by the Euclidean distance taken to the powers $2$,
$3$, $4$, and $10$.
They assume that $\mu$ is the uniform distribution, and test various weights 
and placements for the set $\set{\vc{y}_i}_{i=1}^n$.
When an exact solution cannot be directly determined, they fully discretize 
the problem and use a linear programming solver.

In~\cite{Schmitzer2016a}, Schmitzer works with cost functions
$c = \norm{\cdot}_2^p$ for $p \in (1,\,\infty)$, and applies a form of
adaptive scaling done by ``shielding'' regions:
his method attempts to determine points of influence in order to solve
primarily local problems.
He restricts his examples to $\R^2$.

Solving the semi-discrete problem for the $2$-norm is discussed in 
\cite{Barrett2007a}.\footnote{In 
\cite{Barrett2007a}, the partition of $A$ is called an ``optimal 
coupling.''}
Starting with an alternative form of \eref{MKpde}, taken 
from~\cite{Bouchitte2000a}, \BP{} develop a mixed formulation of the 
Monge-Kantorovich problem, which they solve using a standard finite element 
discretization.

Kitagawa's 2014 paper,~\cite{Kitagawa2014a}, offers a potentially broad
generalization of the Oliker-Prussner algorithm, which works for ground
costs other than $\norm{\cdot}_2^2$, provided those ground costs
satisfy strict conditions, including
Strong Ma-Trudinger-Wang; see also~\cite{Ma2005a}.
His proposals, while densely theoretical, do not include numerics or
an explicit iterative scheme.


As~\cite{Kitagawa2017a} states, the special case $c = \norm{\cdot}_2^2$
has two methods specifically designed for solving semi-discrete problems
directly: the Oliker-Prussner algorithm and the
damped Newton methods proposed in papers like~\cite{Aurenhammer1998a}.
Both rely on some variant of what we call
the Oliker-Prussner approach, described above.
However, approaches developed for fully discrete or continuous transport can also be
applied to the semi-discrete problems, though with varying degrees of effectiveness.
\RU{} apply a discrete linear program solver in~\cite{Ruschendorf2000a}, and the solver
\BP{} use in~\cite{Barrett2007a} was developed for continuous transport.

Discrete methods assume a fully discrete $(X,\,\mu)$ and $(Y,\,\nu)$,
and solve the resulting minimization problem using network flow minimization techniques.
As described in~\cite{Kovacs2015a}, there are over 20 established methods for 
solving such problems, and at least seven software packages capable of handling 
one or more of these methods.

Most approaches to the fully continuous Monge-Kantorovich problem
assume specific ground costs
and solve using techniques developed for 
elliptic partial differential equations, particularly those of the
Monge-Amp\`ere-type:
\begin{equation}\label{MKpde}
-\nabla \cdot (a \nabla u) = f,
\,\text{ where }\,
\abs{\nabla u} \leq 1,\,
a \geq 0,\,
\text{ and }
\abs{\nabla u} < 1 \implies
a = 0.
\end{equation}
If the ground cost function is strictly convex, or otherwise satisfies the
Ma-Trudinger-Wang
regularity 
conditions described in~\cite{Ma2005a}, such problems are well-posed.
To date, the requirements of well-posedness have largely restricted the application of such
continuous methods to well-behaved cost functions such as $\norm{\cdot}_2^2$ or a
regularized Euclidean distance.
Continuous methods currently in use apply finite difference, gradient descent, 
or the iterative Bregman projections (a.k.a.\ Sinkhorn-Knopp) algorithm, all attempting to map $X$ to a fully discretized 
$Y$~\cite{Froese2011a,Jordan1998a,Benamou2015a}.


As we will show, the boundary method offers a new approach to solving semi-discrete transport,
distinct from all of those described above.
By and large, the solution methods
described above only work for a specific fixed cost, usually $c=\norm{\cdot}_2^2$. 
The boundary method quickly solves problems with more general ground costs.
When the ground cost is a $p$-norm, with $p \in (1,\,\infty)$, the boundary method
provides a global rate of convergence that is proportional to the volume of $A$.


\section{Boundary Method}\label{s:boundary_method}
At a high level, the idea behind the boundary method is simple: track only the
boundaries between regions, without resolving the regions' interiors.
To do this in practice and obtain an efficient technique, we must account for
the interplay between discretization, a mechanism for
discarding interior regions, and a fast solver.

At its heart, the boundary method can be viewed as an adaptive refinement technique,
one which focuses on the shared region boundaries.
The method discards interior regions, but a well-chosen initial discretization
prevents any corresponding loss of accuracy.
The boundary method's strategy progressively refines the boundaries between
individual regions $A_i$. Thus, by the method's very nature, any initial configuration
must enclose the boundary in a way that allows it to be distinguished from the region interiors.
The necessary conditions for a well-chosen initial discretization are
presented in \tref{t:goodw1} and discussed in detail in \rref{r:goodw1}.

\subsection{Boundary identity and system of 
equations}\label{s:bound_ident}
For all $i,\,j \in \N_n$ such that $i \neq j$, let
\begin{equation}\label{Aij}
A_{ij} := A_i \cap A_j.
\end{equation}
The \emph{boundary set} is defined as
\begin{equation}\label{bdryset}
B := \bigcup_{1 \leq i < n}\, \bigcup_{i < j \leq n} A_{ij},
\end{equation}
and for each $i \in \N_n$, let the \emph{strict interior} of $A_i$ be defined 
as
\begin{equation}\label{e:oAi}
\mathring{A}_i := A_i \setminus B.
\end{equation}

For all $i,\,j \in \N_n$ such that $i \neq j$, define
$g_{ij} : X \to \R$ as
\begin{equation}\label{gij}
g_{ij}(\vc{x}) := c(\vc{x},\,\vc{y}_i) - c(\vc{x},\,\vc{y}_j).
\end{equation}
By Corollary~\ref{Bnonempt} below, $B \neq \varnothing$ and for each $\vc{x} \in 
B$ there exist $i,\,j \in \N_n$, $i \neq j$, such that $\vc{x} \in A_{ij}$.
Because $\vc{x} \in A_i$, we have
$F(\vc{x}) = a_i - c(\vc{x},\,\vc{y}_i)$, and because $\vc{x} \in A_j$, we have
$F(\vc{x}) = a_j - c(\vc{x},\,\vc{y}_j)$. Combining and rearranging these, we 
get
\begin{equation}\label{e:star}
g_{ij}(\vc{x}) = a_i - a_j\ ,
\quad\quad
\forall \vc{x} \in A_{ij}.
\end{equation}
Thus,
\eref{e:star} implies that $A_{ij}$ is a subset of a level set of $g_{ij}$; the
value $a_i-a_j$ is constant, regardless of which $\vc{x} \in A_{ij}$ is chosen.
Using this information, for each $i,\,j \in \N_n$, $i \neq j$, such that  
$A_{ij} \neq \varnothing$, we can define the constant \emph{shift difference}
\begin{equation}\label{e:aij}
a_{ij} := g_{ij}(\vc{x}_{ij})
\quad\quad
\forall\,
\vc{x}_{ij} \in A_{ij}.
\end{equation}

Given a sufficiently large set of linearly independent equations
of the form given in \eref{e:aij}, one could determine most or all of the
shifts $\set{a_i}_{i=1}^n$.
As we show in \tref{t:N-1exist}, 
it is possible to obtain exactly $(n-1)$ linearly independent
equations of the 
desired form, but a set of $n$ such independent equations does not exist.

Since we know that the set of shifts allows exactly one degree of freedom,
the boundary method's approach is to obtain $(n-1)$ well-chosen $a_{ij}$ 
values, fix one $a_i$, and use linearly independent equations of the form 
given in \eref{e:star} to solve for the remaining $(n-1)$ shifts.
The crucial observation is that for the $a_i$'s, there is 
no need to retain information about interior of the regions.

The \Was{} distance can also be computed without saving region 
interiors.
Once we have determined that $R \subset A_i$ for some region $R$, the 
(partial) \Was{} distance corresponding to $R$ is equal to
\begin{equation}\label{e:P_R}
P_{\scriptscriptstyle{R}} := \int_R c(\vc{x},\,\vc{y}_i) \, 
d\mu(\vc{x}),
\end{equation}
and the total \Was{} distance $P^*$ is equal to the sum of all such partial 
distances $P_{\scriptscriptstyle{R}}$, computed over every $A_i$.

Recognizing these facts, inherent in the shift characterization, inspired both
the boundary method's name and its guiding principles, summarized below:
\begin{center}
\fbox{
  \parbox[c]{2.53in}{\centering
Do \emph{not} solve for the entire transport plan;
\\
rather, identify region boundaries.
  }
}
\end{center}

To illustrate how this principle is implemented, we present the following 
example.

\begin{example}\label{x:steps}
Let $X = Y = [0,\,1]^2$.
Assume $\mu$ is the uniform probability density, so for all Borel sets
$S \subseteq A$, $\mu(S) = \abs{S}$, and that $\nu$ has uniform discrete 
probability density, so $\nu(y_i) = 1/n$ for $1 \leq i \leq n$.
Take $n=5$, with the five points where $\nu$ has nonzero density distributed 
as shown in Figure \ref{f:clrA1}.

Let $c$ be the squared Euclidean norm, $\norm{\vc{y}-\vc{x}}_2^2$.
Suppose a discretization with width $2^{-5}$ is sufficient to provide the 
desired accuracy and that we apply the boundary method with initial width $2^{-4}$.

Assume $\widetilde{P}$ is the \emph{partial transport cost}: the sum cost of transport
over all regions $P_R$ so far, where $P_R$ is defined as in \eref{e:P_R}.
Each iteration consists of two steps. In \aref{a:solve}, we discretize the remaining parts
of $A$ using the given width, and we solve the discrete transport problem. In \aref{a:discard},
we compute the transport cost of all boxes in the interior of each
region, add those costs to $\widetilde{P}$, and discard the computed boxes.
For the discard, remove the transported mass from $\nu$, and remove the
transported boxes from $A$
(so those regions can be safely ignored during any future discretized transport computations).

\fref{f:clrA1} shows the state of the boundary method during the first 
iteration. In \fsubref{f:clrA1}{f:clrA11}, we have just completed 
\aref{a:solve}: the 
discrete transport map has been computed, but we have not identified 
interior points or added anything to the partial transport cost $\widetilde{P}$.
\fsubref{f:clrA1}{f:clrA12} shows the state of the algorithm after 
\aref{a:discard}: the 
interior regions have been identified (shown in gray), the partial 
transport cost has been computed for those regions, giving us $\widetilde{P} = 
0.01387$, and those regions have been discarded.

\begin{figure}[htpb]
  \centering
  \subfloat[Iteration 1, \aref{a:solve}: $\widetilde{P} = 
0.00000$]{%
    \label{f:clrA11}
    \centering
    \resizebox*{0.3\textwidth}{!}{%
    \begin{overpic}[width=\textwidth]{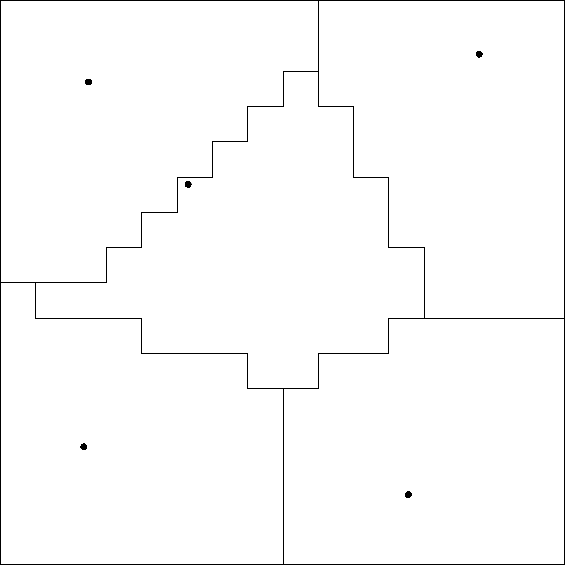}
    \put (16,82) {\scalebox{2.0}{$\vc{y}_0$}}
    \put (86,87) {\scalebox{2.0}{$\vc{y}_1$}}
    \put (34,64) {\scalebox{2.0}{$\vc{y}_2$}}
    \put (15,17) {\scalebox{2.0}{$\vc{y}_3$}}
    \put (73,09) {\scalebox{2.0}{$\vc{y}_4$}}\end{overpic}}
  }%
  \quad\quad\quad\quad%
  \subfloat[Iteration 1, \aref{a:discard}: $\widetilde{P} = 0.01387$]{%
    \label{f:clrA12}
    \centering
    \resizebox{0.3\textwidth}{!}{%
    \begin{overpic}[width=\textwidth]{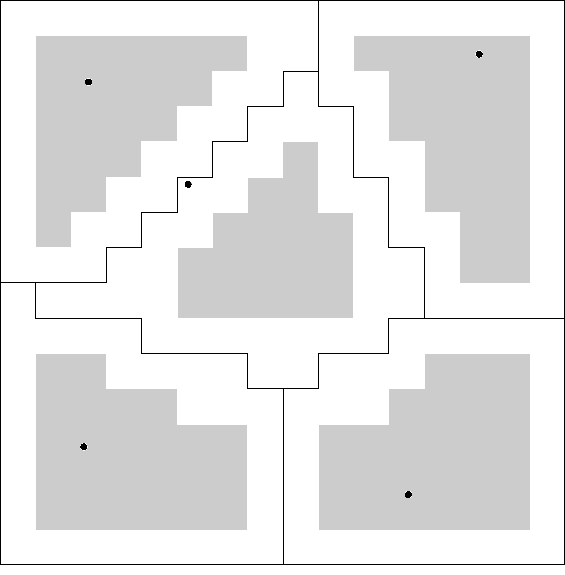}
    \put (16,82) {\scalebox{2.0}{$\vc{y}_0$}}
    \put (86,87) {\scalebox{2.0}{$\vc{y}_1$}}
    \put (34,64) {\scalebox{2.0}{$\vc{y}_2$}}
    \put (15,17) {\scalebox{2.0}{$\vc{y}_3$}}
    \put (73,09) {\scalebox{2.0}{$\vc{y}_4$}}\end{overpic}}
  }
\caption{Iteration 1 of \xref{x:steps}: $w_1 = 
2^{-4}$, computed regions in gray}\label{f:clrA1}
\end{figure}

\fref{f:clrA2} shows the state of the boundary method algorithm during the 
second iteration.
Here, the regions eliminated in Iteration 1 are shown in a darker gray, to
distinguish new interiors from those previously removed.
In \fsubref{f:clrA2}{f:clrA21}, \aref{a:solve} has just been completed. As can 
be seen by 
comparing \fsubref{f:clrA1}{f:clrA12} to \fsubref{f:clrA2}{f:clrA21}, the 
boundary and interior regions 
are the same ones that we had at the end of the first iteration, but refining 
the boundary set to width $w_2 = 2^{-5}$ allows us to compute a more refined 
transport map.
Since the regions in gray were discarded at the end of Iteration 1 
\aref{a:discard},
they are not part of the discrete transport solution computed during Iteration 
2.
Because \aref{a:solve} does not add to the identified interior 
regions, the partial \Was{} distance $\widetilde{P}$ is also unchanged from 
\fsubref{f:clrA1}{f:clrA12}.

After \aref{a:discard} of the second iteration, shown in 
\fsubref{f:clrA2}{f:clrA22}, more of the interiors have been identified.
The partial transport cost shows a 
corresponding increase: we now have $\widetilde{P} = 0.02898$.
Because we have achieved our desired refinement, a width of $2^{-5}$, we
end the iterative process.

We have not computed any transport cost for the white areas remaining in 
\fsubref{f:clrA2}{f:clrA22}. Hence, $\widetilde{P}$ is strictly less than the
actual transport cost $P^*$. We may want to perform further computations
on those white areas in order to approximate the remaining transport cost and
calculate an error bound for our approximation.

\begin{figure}[htpb]
  \centering
  \subfloat[Iteration 2, \aref{a:solve}: $\widetilde{P} = 0.01387$]{%
    \label{f:clrA21}
    \centering
    \resizebox{0.3\textwidth}{!}{%
    \begin{overpic}[width=\textwidth]{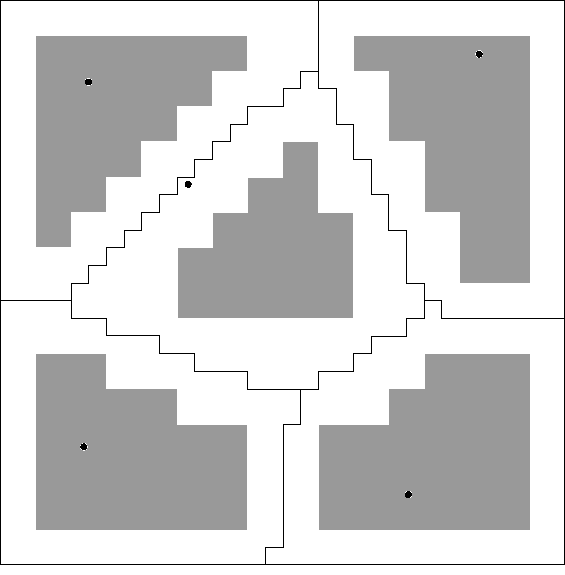}
    \put (16,82) {\scalebox{2.0}{$\vc{y}_0$}}
    \put (86,87) {\scalebox{2.0}{$\vc{y}_1$}}
    \put (34,64) {\scalebox{2.0}{$\vc{y}_2$}}
    \put (15,17) {\scalebox{2.0}{$\vc{y}_3$}}
    \put (73,09) {\scalebox{2.0}{$\vc{y}_4$}}\end{overpic}}
  }%
  \quad\quad\quad\quad%
  \subfloat[Iteration\ 2, \aref{a:discard}: $\widetilde{P} = 0.02898$]{%
    \label{f:clrA22}
    \centering
    \resizebox{0.3\textwidth}{!}{%
    \begin{overpic}[width=\textwidth]{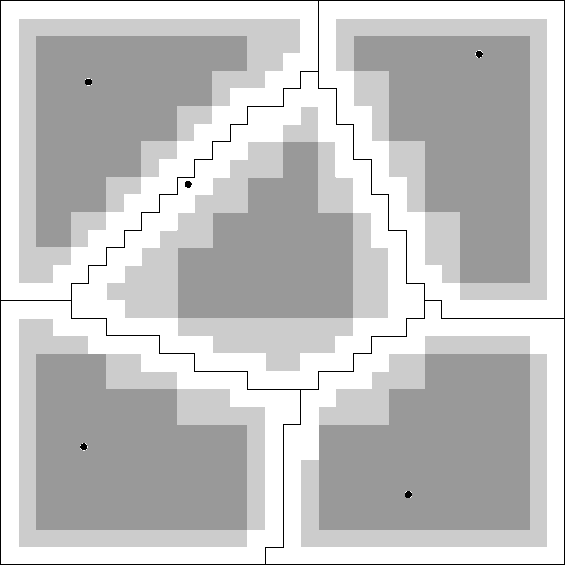}
    \put (16,82) {\scalebox{2.0}{$\vc{y}_0$}}
    \put (86,87) {\scalebox{2.0}{$\vc{y}_1$}}
    \put (34,64) {\scalebox{2.0}{$\vc{y}_2$}}
    \put (15,17) {\scalebox{2.0}{$\vc{y}_3$}}
    \put (73,09) {\scalebox{2.0}{$\vc{y}_4$}}\end{overpic}}
  }%
\caption{Iteration $2$ of \xref{x:steps}: $w_2 = 
2^{-5}$, computed regions in gray}\label{f:clrA2}
\end{figure}
\end{example}

\subsection{The boundary method}\label{s:bm}
We will now formalize the process described in \xref{x:steps}.
As described below, the boundary method generates a grid $A^r$ over 
the unevaluated region of $A$, and uses it to determine the subgrid $B^r$ 
containing the boundary set $B$.
This subgrid is determined by finding an optimal transport solution 
from the grid $A^r$ to the point set $\set{\vc{y}_i}_{i=1}^n$.

Although not strictly necessary, we will restrict ourselves to $A 
= [0,l]^d$ and apply a Cartesian grid over that region.
At the $r$-th refinement level of the algorithm, the grid will thus consist of 
a 
collection of boxes with width $w_r$ in each dimension of our discretization.
By a slight abuse of notation, we use $\vc{x}^r$ to refer to such a box, 
centered at the point $\vc{x}$.
Thus, $\mu(\vc{x}^r)$
refers to the $\mu$-measure of the box of width 
$w_r$ centered at $\vc{x}$.

\emph{Neighboring boxes} are those with center points 
that differ by no more than one unit in any discretization index.
The set of \emph{neighbors} of $\vc{x}$ is denoted $N(\vc{x})$ (defined 
in \eref{e:neighb}, below).
Because regions of $\mu$-measure zero need not be transported to any 
particular $\vc{y}_i$, boxes of positive weight that are adjacent to such 
regions are always retained.
We refer to such a box as an \emph{edge box}.
Thus, the set of edge boxes is
\begin{equation}\label{e:edgeAr}
\edge(A^r) :=
\set{
\vc{x} \in A^r
\mid
\mu(\vc{x}) > 0
\,
\text{ and }
\,
\exists\,
\vc{x}_n \in N(\vc{x})
\,
\text{ such that }
\,
\mu(\vc{x}_n) = 0
}.
\end{equation}
Because $A$ contains the support of $\mu$, every box of positive mass that is
adjacent to the boundary of $A$ is an edge box.

A box whose neighbors and itself all have positive measure is referred to as an 
\emph{internal box}.
The set of internal boxes is
\begin{equation}\label{e:intAr}
\inter(A^r) :=
\set{
\vc{x} \in A^r
\mid
\mu(\vc{x}) > 0
\,\,
\text{ and }
\,\,
\mu(\vc{x}_n) > 0
\,\,
\text{ for all }
\,\,
\vc{x}_n \in N(\vc{x})}.
\end{equation}
Boxes of $\mu$-measure zero are not part of $\edge(A^r)$ or 
$\inter(A^r)$ and they are discarded when the optimal transport problem is 
solved.
We need not be concerned about losing a region $A_i$ 
due to this discard process, since this would imply $\mu(A_i)=0$ 
(and hence $\nu(\vc{y}_i) = 0$, which contradicts the conditions in
\sref{s:semi-discrete}).

Region interiors are identified by comparing the destination of each $\vc{x} 
\in 
\inter(A^r)$ to the destinations of its neighbors.
Edge boxes are never considered part of a region interior, so they are passed 
directly to $B^r$.

In order to remove identified region interiors, we also maintain a running 
total of the untransported mass, given by \emph{partial measure} $\tilde{\nu}$.
To preserve the balance of the transport problem, each time a region 
$\vc{x}^r$ is transported from $A$ to $\vc{y}_i$, the remaining amount that can 
be transported to $\vc{y}_i$, $\tilde{\nu}(\vc{y}_i)$, must be reduced by 
$\mu(\vc{x}^r)$.

\begin{algorithm}
\centering
\fbox{
  \parbox[c]{0.8\textwidth}{
\begin{center}
\textbf{Boundary method algorithm}
\end{center}
\begin{algolist}
\setcounter{algolisti}{-1}
\item
Set $\widetilde{P} = 0$, $\tilde{\nu} = \nu$, and $r = 1$.
Create $A^r = A^1$ from $A$.
\item\label{a:solve}
Solve the discretized transport solution.
\item\label{a:discard}
For each $\vc{x} \in \inter(A^r)$:

Are the neighbors of $\vc{x}$ all transported to the same $\vc{y}_i$?
\begin{itemize}
\item If so, then $\vc{x}^r$ is in the interior of $A_i$:
\begin{itemize}
\item {[optional]} Add $\displaystyle{\int_{\vc{x}^r} c(\vc{z},\,\vc{y}_i) 
\,d\mu(\vc{z})}$ 
to $\widetilde{P}$.
\item Reduce the value of $\tilde{\nu}(\vc{y}_i)$ by $\mu(\vc{x}^r)$.
\item Remove $\vc{x}$ from $\inter(A^r)$.
\end{itemize}
\end{itemize}
\end{algolist}
The sets $\edge(A^r)$ and the reduced set $\inter(A^r)$ combine to form $B^r$.
\begin{algolist}
\setcounter{algolisti}{2}
\item\label{a:refine}
Is the desired refinement reached?
\begin{itemize}
\item
If not:
\begin{itemize}
\item
Refine $B^r$ to create $A^{r+1}$, increment $r$, and go to \aref{a:solve}.
\end{itemize}
\end{itemize}
\end{algolist}
Optionally, once the desired refinement level is reached:
\begin{algolist}
\setcounter{algolisti}{3}
\item\label{a:diffs}
Use $B^r$ to identify $(n-1)$ appropriate shift differences $\set{a_{ij}}$
\\ and solve for the shifts $\set{a_i}_{i=1}^n$.
\item\label{a:wass}
Use $\widetilde{P}$ and $B^r$ to approximate $W_1(\mu,\,\nu)$.
\end{algolist}
}}
\end{algorithm}

We can approximate the \Was{} distance $P^*$ by generating a running total over 
region interiors: $\widetilde{P}$.
This $\widetilde{P}$ is an increasing function of $r$, 
and for all $r$, $P^* \geq \widetilde{P}$.
The \Was{} distance over any remaining boundary region is evaluated at 
completion.

\begin{rem}\label{r:approx}
Further approximations may be required for a truly general algorithm.
Depending on $\mu$, it may be necessary to approximate the mass of each box,
$\mu(\vc{x}^r)$.
Depending on $\mu$ and $c$, the \Was{} distance over each box, given by
$\displaystyle{\int_{\vc{x}^r} c(\vc{z},\,\vc{y}_i) \,d\mu(\vc{z})}$, may also 
require approximation.
However, in this work we assume that the integrals can be computed exactly.
In practice, this is not a significant limitation.
Most numerical applications focus on the exactly-computable cases
where $\mu$ is uniform and $c$ is the Euclidean or squared-Euclidean distance.
Furthermore, as we show in \sref{s:numerics}, the set of exactly-computable
options is quite large.
\end{rem}

\subsubsection{Step (1): solving the discrete optimal transport 
problem}\label{s:ga}
The proofs in \sref{s:math} assume the discrete solver is exact, but in
practice we achieve good results using solvers whose error satisfies
reasonable bounds.
Thus, the ideal discrete algorithm should be \emph{fast}, have controlled 
error, 
and possess reasonable scaling properties.
To satisfy these requirements, and to bypass the shortcomings of standard 
discrete approaches, 
we have turned to the distributed relaxation methods known as 
auction algorithms;
see \cite{Bertsekas1998a} and \cite{Bertsekas1993a}.
(As it turns out, there are natural connections between auction algorithms
and the Oliker-Prussner algorithm for semi-discrete transport; see 
\cite{Merigot2013a}
for details).

We chose to apply a new auction algorithm, the 
{\emph{general auction}}, which we developed and presented in~\cite{Walsh2016a}.
The general auction is so named because it is based directly on the (more 
general) real-valued transport problem, rather than the integer-valued 
assignment problem which forms the foundation of other auction algorithms.
As described in~\cite{Walsh2016a}, it offers significant performance advantages 
over other auction algorithms.
Public domain \CC{} software implementing the general auction can be found 
on the Internet at~\cite{Walsh2016b}.

\subsubsection{Step (4): computing the shifts}\label{s:diffs}
Once we have reached a desired level of refinement for the boundary, we can
use the set $B^r$ to identify $(n-1)$ shift differences $a_{ij}$.
Finding the shift differences is not necessary once we have the boundary
(which is why Step (4) is optional), but the shift differences
allow one to reconstruct the entire transport map.

By completing Step (4), one can reduce the transport map
in $\R^d$ to a set
of $n$ real numbers $a_i$, greatly reducing storage requirements.
Also, building the reconstructed transport map, and comparing the value of each
$\mu(A_i)$ to its corresponding $\nu(\vc{y}_i)$, effectively
evaluates the actual (vs.\ worst case) error associated with
the boundary method's solution.

It is also worth considering that the
exact shifts $\set{a_i}_{i=1}^n$ correspond to a transport map giving the 
exact optimal solution of our semi-discrete problem.
The approximated shifts $\set{\tilde{a}_i}_{i=1}^n$, unless generating the 
same shift differences, correspond to a transport map giving the exact optimal 
solution to a \emph{different} semi-discrete problem, one whose measure $\nu$
at each $\vc{y}_i$, $i \in \N_n$, corresponds to the value of 
$\mu(\tilde{A}_i)$.
Hence, $\abs{\mu(\tilde{A}_i) - \nu(\vc{y}_i)}$ is the error in measure when
approximating $A_i$ by $\tilde{A}_i$.

\subsubsection{Step (5): approximating the Wasserstein 
distance}\label{s:wass}
Because some applications focus on determining the transport map,
rather than the Wasserstein distance, Step (5) is optional.
One could also skip the computation of $\widetilde{P}$ in Step (2),
since the Wasserstein distance can be computed in full
using only the transport map defined by the boundary set.
However, we find it convenient to compute as much of the distance
as possible within the boundary method algorithm,
establishing $\widetilde{P}$ one box at a time during Step (2).
By the time we reach \aref{a:wass}, the partial \Was{} distance $\widetilde{P}$ 
includes the exact cost of all the identified interior regions,
and all that remains is to determine the cost of the regions associated 
with $B^r$.

\section{Mathematical support}\label{s:math}
In this section, we provide mathematical support for the boundary 
method, assuming that all computations are solved exactly: both the discrete 
optimal transport problems handled by the general auction and the 
determinations of mass and \Was{} distance for individual boxes (see 
\rref{r:approx}).
We present three types of results:
on the shift characterization, on our system of equations,
and, finally, on the boundary method itself.

\subsection{Semi-discrete optimal 
transport and the shift characterization}\label{s:oc_proofs}
Here we examine the features of the shift characterization, defined in 
\sref{s:oc_form}, and consider what they can tell us about the semi-discrete 
optimal transport problem itself.
While many of these results can be found in other works 
(e.g.,~\cite{Gangbo1996a}),
detailing them fixes notation and sets the stage for the original theorems
developed in the following sections.

First, in Lemmas \ref{Lemma1} and \ref{Lemma2},
we develop theoretical support for the boundary method.

\begin{lemma}\label{Lemma1}
Let $a_i$ and $A_i$ be defined as in \dref{ShiftChar}.
Fix $i \in \N_n$. If $\vc{x} \in A_{i}$ and $j \in \N_n$, $j 
\neq i$, then the following hold:
\begin{align}\label{th:inAi}
g_{ij}(\vc{x}) &\leq a_i - a_j,
\\
\label{th:inAij}
g_{ij}(\vc{x}) &= a_i - a_j
\quad\quad
\iff
\quad\quad
\vc{x} \in A_{ij},\text{ and }
\\
\label{th:inAi-Aj}
g_{ij}(\vc{x}) &< a_i - a_j
\quad\quad
\iff
\quad\quad
\vc{x} \in A_{i} \setminus A_{j},
\end{align}
where $g_{ij}$ is defined in \eref{e:star} and $A_{ij}$ in \eref{Aij}.
\end{lemma}

\begin{proof}
Let us show \eref{th:inAi}.
By the definitions of $A_i$ and $F$,
\begin{equation*}
a_i - c(\vc{x},\,\vc{y}_i) = F(\vc{x})
\geq a_j - c(\vc{x},\,\vc{y}_j).
\end{equation*}
Rearranging terms gives
\begin{equation*}
c(\vc{x},\,\vc{y}_i) - c(\vc{x},\,\vc{y}_j) \leq a_i - a_j.
\end{equation*}

To show \eref{th:inAij}, first note that \sref{s:bound_ident} already 
explains how $\vc{x} \in A_{ij}$ implies $g_{ij}(\vc{x}) = a_i - a_j$.
Consider the converse:
Assume that $g_{ij}(\vc{x}) = a_i - a_j$.
Rewriting, we find that
$a_j - c(\vc{x},\,\vc{y}_j) = a_i - c(\vc{x},\,\vc{y}_i)
= F(\vc{x})$, with $F$ defined in \eref{Fdef}.
This implies $\vc{x} \in A_j$, and since $\vc{x} \in A_i$, therefore 
$\vc{x} \in A_{ij}$.
\eref{th:inAi-Aj} is a consequence of Equations 
\eqref{th:inAi} and \eqref{th:inAij}.
\end{proof}

\begin{lemma}\label{Lemma2}
Let $a_i$ and $A_i$ be defined as in \dref{ShiftChar} and $A_{ij}$ as in 
\eref{Aij}.
Assume
$c$ satisfies the triangle inequality.
For all $i,\,j \in \N_n$,
$i \neq j$,
\begin{itemize}
\item[(a)]
If $c(\vc{y}_i,\,\vc{y}_j) = a_i-a_j$, then $A_j \subseteq A_{ij}$.
\item[(b)]
If $c(\vc{y}_i,\,\vc{y}_j) < a_i-a_j$, then $A_j = \varnothing$.
\end{itemize}
\end{lemma}
\begin{proof}
For Part (a), because $c$ satisfies the triangle inequality, for all 
$\vc{x}\in A$,
\begin{align}\label{TriangIneq}
\begin{split}
c(\vc{x},\,\vc{y}_i) &\leq c(\vc{x},\,\vc{y}_j) + c(\vc{y}_i,\,\vc{y}_j) \\
c(\vc{x},\,\vc{y}_i) &\leq c(\vc{x},\,\vc{y}_j) + a_i - a_j \\
a_j - c(\vc{x},\,\vc{y}_j) &\leq a_i - c(\vc{x},\,\vc{y}_i).
\end{split}
\end{align}
Suppose $\vc{x} \in A_j$. Then
$a_i - c(\vc{x},\,\vc{y}_i) \geq a_j - c(\vc{x},\,\vc{y}_j) = F(\vc{x})$, by 
\eref{Fdef}. 
Because $F$ is defined as the maximum such difference, this implies $a_i - 
c(\vc{x},\,\vc{y}_i) = F(\vc{x})$, and so $\vc{x} \in A_i$. Further, since 
$\vc{x}$ is an element of $A_i$ and $A_j$, $\vc{x} \in A_{ij}$.
Therefore, $A_j \subseteq A_{ij}$.

To show (b), note that \eqref{TriangIneq} now gives
$a_j - c(\vc{x},\,\vc{y}_j) < a_i - c(\vc{x},\,\vc{y}_i)$.
Hence, for all $\vc{x} \in A$,
$F(\vc{x}) \geq a_i - c(\vc{x},\,\vc{y}_i)
> a_j - c(\vc{x},\,\vc{y}_j)$.
Therefore, $A_j = \varnothing$.
\end{proof}

\begin{lemma}\label{l:Fcont}
Let $F(\vc{x})$ be defined by \eref{Fdef}.
If the ground cost function $c(\vc{x},\,\vc{y})$ is continuous on $X \times Y$, 
then
$F(\vc{x})$ is a continuous function of $\vc{x}$.
\end{lemma}
\begin{proof}
Assume $c$ is defined as a continuous function in $X \times Y$. 
Thus, for all $i \in \N_n$, $a_i - c\vc({x},\,{y}_i)$ is a continuous 
function of $\vc{x}$. Since $F$ is the maximum of a finite set of 
continuous functions, $F$ is itself a continuous function of $\vc{x}$.
\end{proof}

\begin{definition}[$F$ induces a $\mu$-partition of $A$]\label{d:part_A}
Let $F$ be as defined in \eref{Fdef}, and the sets $A_i$ as defined in 
\eref{Aidef} for $i \in \N_n$.
Then one says $F$ \emph{induces a $\mu$-partition} of the 
set $A$ if
\begin{enumerate}
\item $\mu(A) < \infty$,
\item for all $i,\,j \in \N_n$, $i\neq j$, $\mu(A_{ij})=0$ (for $A_{ij}$ as 
defined in \eref{Aij}),
\item
$\sum_{i=1}^n \mu(A_i) = \mu(A)$, and
\item for all $i \in \N_n$, $\mu(A_i) = \nu(\vc{y}_i) > 0$.
\end{enumerate}
\end{definition}

\begin{lemma}\label{l:muPartition}
Suppose one has a semi-discrete transport problem, as described in 
\sref{s:semi-discrete}.
Let $F$ be as defined in \eref{Fdef}, the sets $A_i$ as defined in 
\eref{Aidef} for $i \in \N_n$, and
$B$ as defined in \eref{bdryset}.
Then $F$ induces a $\mu$-partition of $A$ if and only if 
$\mu(B) = 0$.
\end{lemma}
\begin{proof}
If $F$ induces a $\mu$-partition of $A$, by \dref{d:part_A}, $\mu(B) = 0$.
For the converse, assume $F$ and the sets $A_i$ are defined as given, and
let $A_{ij}$ be defined by \eref{Aij}.
Because $\mu$ is a probability density function, $\mu(A) = 1 < \infty$.
Because $\mu$ is a non-negative measure, $\mu(B) = 0$ implies that, for all 
$i,\,j \in \N_n$, $i \neq j$, $\mu(A_{ij}) = 0$.

For any $\mu$-measurable set $S \subseteq X$, $S = S_1 \cup S_2$,
\begin{equation*}
\mu(S_1) + \mu(S_2) = \mu(S) + \mu (S_1 \cap S_2),
\end{equation*}
and since $\mu(X) < \infty$,
\begin{equation*}
\mu(S) = \mu(S_1) + \mu(S_2) - \mu(S_1 \cap S_2).
\end{equation*}
Proceeding inductively, it follows that
\begin{equation*}
S = \bigcup_{i=1}^n S_i, \text{ all $\mu$-measurable }
\quad \implies \quad
\mu(S) = \sum_{i=1}^n \mu(S_i)
- \sum_{i=1}^n \sum_{\substack{j=1\\j \neq i}}^n \mu(S_i \cap S_j).
\end{equation*}
Thus,
\begin{align*}
1 = \mu(A) &= \sum_{i=1}^n \mu(A_i)
- \sum_{i=1}^n \sum_{\substack{j=1\\j \neq i}}^n \mu(A_{ij})
= \sum_{i=1}^n \mu(A_i).
\end{align*}
For all $i,\,j \in \N_n$, $i \neq j$, $\mu(A_i \cap A_j) = 0$, and therefore
$\mu(A_i) = \nu(\vc{y}_i)$.
\end{proof}
\begin{rem}
Instances of $\mu(B)>0$ appear quite often, though (as we will show)
$\mu(B)=0$ for the $p$-norm cost functions we have assumed. For an
example of $\mu(B)>0$ in the literature, see Figure 37 
of~\cite{Aurenhammer1991a}.
We include a nearly identical example as \fsubref{f:exact}{f:badManhattan} of 
our paper,
along with a discussion of this behavior.
\end{rem}

Given our definition of the semi-discrete problem in \sref{s:semi-discrete}, 
Corollary 4 of \cite{Cuesta1993a} provides a sufficient condition for 
the existence of a Monge solution that is unique $\mu$-a.e.
For convenience, we restate their conclusion here, as the following:
\begin{theorem}\label{t:CuestaAlbertos}
Given the definition of $g_{ij}$ in \eref{e:star},
suppose that the support of $\nu$ is finite, $c$ is continuous, $\mu$ is tight,
and
\begin{equation}\label{e:suffUniqueMonge}
\mu\left( \set{\vc{x} \in A \mid g_{ij}(\vc{x}) = k } \right) = 0
\quad\quad
\forall\, i,\,j \in \N_n,\,i\neq j,
\quad\quad
\forall\, k \in \R.
\end{equation}
Then there exists an optimal transport map, $T:X \to Y$, that solves the Monge 
problem, and $T$ is $\mu$-a.e.\ unique.
\end{theorem}
This condition leads directly to the following theorem.

\begin{theorem}\label{t:unique_solution}
A semi-discrete transport problem, as described in \sref{s:semi-discrete}, has 
an associated transport map $T$, 
a function $F$, as described in \eref{Fdef}, 
and sets $\set{A_i}_{i=1}^n$, as described in \eref{Aidef}, such that for all 
$\vc{x} \in A$,
\begin{equation}
\vc{x} \in \mathring{A}_i \text{ for some } i \in \N_n
\quad\quad \implies \quad\quad
T(\vc{x}) = \vc{y}_i,
\end{equation}
where $\mathring{A}_i$ is the strict interior of $A_i$, as defined in 
\eref{e:oAi}.
In other words,
$F$ induces a $\mu$-partition of $A$ and
$T$ agrees with $F$ on $A \setminus B$.
Furthermore, $T$ is unique $\mu$-a.e.
\end{theorem}

\begin{proof}
Let $A_{ij}$ be defined as given in \eref{Aij}, $B$ as given in \eref{bdryset},
and $g_{ij}$ as in \eref{e:star}. 
Consider the requirements given in \sref{s:semi-discrete}.
\cref{c:nuFinite} ensures that $\nu$ is finite, and
\cref{c:pNorm} implies $c$ is continuous.
We know that $A \subseteq \R^d$, so $A$ is a Polish space,
and \cref{c:convComp} assures us that $A$ is compact. Because every probability
measure on a compact Polish space is tight\footnote{see e.g.\ Theorem
3.2 of~\cite[p.\ 29]{Parthasarathy1967a}}, $\mu$ must be tight.
Because \cref{c:pNorm} requires that the ground cost is equal to a $p$-norm
with $p \in (1,\,\infty)$,
\begin{equation}\label{e:0wrtLm}
\abs*{ \set{\vc{x} \in A \mid g_{ij}(\vc{x}) = k }} = 0
\quad\quad
\forall\, i,\,j \in \N_n,\,i\neq j,
\quad\quad
\forall\, k \in \R.
\end{equation}
By \cref{c:absCont}, $\mu$ is absolutely continuous, and so
\begin{equation*}
\mu\left( \set{\vc{x} \in A \mid g_{ij}(\vc{x}) = k } \right) = 0
\quad\quad
\forall\, i,\,j \in \N_n,\,i\neq j,
\quad\quad
\forall\, k \in \R,
\end{equation*}
as required by \eref{e:suffUniqueMonge}
(see~\cite{Walsh2017a} for another argument).
Therefore, the conditions of \tref{t:CuestaAlbertos} are satisfied.

Let the function $F$ and sets $\set{A_i}_{i=1}^n$ be as described in 
\dref{ShiftChar}.
For any $i,\,j\in\N_n$, $i\neq j$, $A_{ij} \subseteq \set{\vc{x} \in A \mid
g_{ij}(\vc{x}) = k }$ for some fixed $k \in \R$. Hence,
it follows that $\mu(B)=0$, and thus, by \lref{l:muPartition}, $F$ induces a 
$\mu$-partition of $A$. Therefore, we can construct a transport plan $T$ that 
satisfies the semi-discrete problem and agrees with $F$ on $A \setminus B$. 
Furthermore, by \tref{t:CuestaAlbertos}, $T$ is unique $\mu$-a.e.
\end{proof}


\begin{rem}
A close reading of the text of \tref{t:unique_solution} reveals that the 
guarantee of $\mu(B)=0$ derives
directly from the fact that $\abs{B}=0$; see \eref{e:0wrtLm}. Absolute 
convergence does the rest. In practice, this means that the boundary
method forces a unique transport map on all of $A$, even regions where $\mu$ 
vanishes and any other map would
achieve the same Wasserstein measure. For an example of this, see \fref{f:zero}.
This behavior stems from a natural (unstated) corollary to 
\tref{t:unique_solution}: the boundaries identified by our method
are a.e.\ unique with respect to the Lebesgue measure.
The convexity of $A$ is required to guarantee the existence of the requisite 
boundaries.
Otherwise, the network of regions might not form a connected graph.
\end{rem}

\subsection{Existence of linearly independent boundary 
equations}\label{s:exist}
To prove the existence of $(n-1)$ linearly independent equations
of the form shown in \eref{e:aij}, we will investigate the structure of the 
boundary set using a connected graph.\footnote{For a different approach, where
the cost is the squared-Euclidean distance, see~\cite{Aurenhammer1987a}.}

\begin{definition}\label{graphG}
Assume the definition of $A_{ij}$ given in \eref{Aij}.
Let $G$ be a graph with $n$ vertices $v_1,\,\ldots,\,v_n$. The edge 
$(v_i,\,v_j)$ is contained in the edge set of $G$ if and only if $A_{ij}$ is 
non-empty.
We refer to $G$ as the \emph{adjacency graph} of our transport problem.
\end{definition}

\begin{lemma}\label{l:Gconnected}
Let $G$ be defined as given in \dref{graphG}.
If the set $A$ is convex and compact, then
$G$ is a connected graph.
\end{lemma}
\begin{proof}
Assume to the contrary that $G$ is not a connected graph. Then we can write 
$G$ as the union of two disjoint nonempty subgraphs, $G = G_1 \cup G_2$, such 
that no 
vertex $v_1$ in $G_1$ has a path connecting it to any vertex 
$v_2$ in $G_2$.

Construct
\begin{equation*}
\tilde{A}_1 := \bigcup_{v_i \in G_1} A_i
\quad\quad \text{ and } \quad\quad
\tilde{A}_2 := \bigcup_{v_j \in G_2} A_j,
\end{equation*}
where each subset is defined as in \eref{Aidef}.
Since $G_1 \neq \varnothing$ and $G_2 \neq \varnothing$,
$\tilde{A}_1 \neq \varnothing$ and $\tilde{A}_2 \neq \varnothing$.
Because $G_1$ and $G_2$ are disjoint, and no paths connect them, it follows 
that $\tilde{A}_1 \cap \tilde{A}_2 = \varnothing$. Since the union of $G_1$ and 
$G_2$ is $G$,
$\tilde{A}_1 \cup \tilde{A}_2 = A$.

Suppose $A_i \subseteq \tilde{A}_1$, $A_j \subseteq \tilde{A}_2$. Then $A_{ij} = 
\varnothing$.
Because $A$ is a compact set,
$A$ is a closed and bounded, and hence the definition given in \eref{Aidef} 
implies that $A_i$ and $A_j$ must each also be closed and bounded.
Thus, $A_i$ and $A_j$ are disjoint compact sets in the Hausdorff space 
$\R^d$.
This implies $A_i$ and $A_j$ are separated by some positive distance 
$\epsilon_{ij}$.
Because this is true for all $A_i \subseteq \tilde{A}_1$ and $A_j \subseteq 
\tilde{A}_2$, there 
exists $\epsilon > 0$, the minimum over all such 
$\epsilon_{ij}$.

Let $\vc{x}_1 \in \tilde{A}_1$, $\vc{x}_2 \in \tilde{A}_2$, and for all $t \in 
[0,\,1]$, define
\begin{equation*}
\vc{x}_t = (1-t)\vc{x}_1 + t\vc{x}_2.
\end{equation*}
Because $\epsilon > 0$, there exists $(t_0,\,t_1) \subseteq [0,\,1]$, 
$\abs{t_1 - t_0} \geq \epsilon$, such that
$t \in (t_0,\,t_1)$ implies $\vc{x}_t \notin \tilde{A}_1 \cup \tilde{A}_2 = A$.
This contradicts the convexity of $A$.
Hence, $G$ is connected.
\end{proof}

\begin{corollary}\label{Aijnonempt}
Assume $n \geq 2$ and let $A_{ij}$ be defined by \eref{Aij}.
If $i \in \N_n$, there exists $j \in 
\N_n$, such that $j \neq i$ and $A_{ij} \neq \varnothing$.
\end{corollary}
\begin{proof}
Assume the contrary for some $i$, and apply \dref{graphG}. Since $n \geq 2$, $G$ 
includes at least two 
vertices, and $v_i$ is disconnected from the rest of $G$, which contradicts 
\lref{l:Gconnected}.
\end{proof}

\begin{corollary}\label{Bnonempt}
Let $A_{ij}$ be defined by \eref{Aij} and $B$ by \eref{bdryset}.
If $n \geq 2$, then the boundary set $B$ is nonempty, and for each $\vc{x} \in 
B$, there exist 
$i,\,j \in \N_n$ such that $i \neq j$ and $\vc{x} \in A_{ij}$.
\end{corollary}
\begin{proof}
This follows from Corollary \ref{Aijnonempt} and the 
definition of $B$ in \eref{bdryset}.
\end{proof}

\begin{lemma}\label{l:nocycles}
Assume a shift characterization, as described in \dref{ShiftChar}, where $n \geq 
2$ and the
shifts $\set{a_i}_{i=1}^n$ are unknown.
Let $G$ be the adjacency graph of the transport problem given in \dref{graphG}, 
and
let $H$ be a 
subgraph of $G$ that includes all $n$ vertices.
Define the system of equations
\begin{equation}\label{systemS}
S := \set{ a_i - a_j = a_{ij} \mid (v_i,\,v_j) \in
\text{ the edge set of } H },
\end{equation}
where each $a_{ij}$ is given by some constant.
The system of equations $S$ is linearly independent with respect to the 
shifts $\set{a_i}_{i=1}^n$ if and only if $H$ contains no cycles.
\end{lemma}
\begin{proof}
$(\Longrightarrow)$
Suppose $H$ contains the cycle 
$(v_{i_1},\,v_{i_2},\,\ldots,\,v_{i_k},\,v_{i_1})$.
Then $S$ contains the linear system
\begin{equation*}
M
\begin{bmatrix}
a_{i_1} \\ \vdots \\ a_{i_{k-1}} \\ a_{i_k}
\end{bmatrix}
=
\begin{bmatrix}
a_{i_1i_2} \\ \vdots \\ a_{i_{k-1}i_k} \\ a_{i_ki_1}
\end{bmatrix}
,
\quad\quad
\text{ where }
M =
\begin{bmatrix}
\phantom{-}1 & -1 &       &    &    \\
   &    &\ddots &    &    \\
   &    &       &  \phantom{-}1 & -1 \\
-1 &    &       &    &  \phantom{-}1 \\
\end{bmatrix}.
\end{equation*}
Because $\det(M) = 0$, we know $S$ is linearly dependent.

$(\Longleftarrow)$ Suppose instead that $S$ is linearly dependent.
Given the form of the equations in $S$, we can assume without loss of 
generality that $S$ contains the equations
$a_{i_ji_{j+1}} = a_{i_j} - a_{i_{j+1}}$,
$\forall j \in \N_{k-1}$,
and that $a_{i_1i_k} = a_{i_1} - a_{i_k}$ is also in $S$.
By the definition of $S$, these equations imply that the edges 
$(v_1,\,v_2),\,(v_2,\,v_3),\,\ldots,\,(v_{k-1},\,v_k)$, and $(v_k,\,v_1)$ are 
contained in $H$.
Together, these edges generate the cycle 
$(v_{i_1},\,v_{i_2},\,\ldots,\,v_{i_k},\,v_{i_1})$, so $H$ contains at least 
one cycle.
\end{proof}

\begin{theorem}\label{t:N-1exist}
Assume a shift characterized problem, as described in \dref{ShiftChar}, where $n 
\geq 2$
and the shifts $\set{a_i}_{i=1}^n$ are unknown.
Then
there exists at least one system of exactly $(n-1)$ equations of the form $a_i 
- a_j = a_{ij}$ that is linearly independent with respect to the
set of shifts 
$\set{a_i}_{i=1}^n$, with each $a_{ij}$ constant.
No system of $n$ independent equations exists.
\end{theorem}
\begin{proof}
Let $G$ be as given in \dref{graphG}.
Because $G$ is a connected graph, we can always create a spanning tree $H$ that 
is a subgraph of $G$. Let $S$ be the corresponding set of linear equations, 
defined as described in \eqref{systemS}. As a spanning tree, $H$ contains 
$(n-1)$ edges and $H$ has no cycles, so by \lref{l:nocycles}, we know $S$
contains exactly $(n-1)$ linearly independent equations.

Suppose a set $S$ of $n$ linearly independent equations exists, all of the 
form 
$a_i - a_j = a_{ij}$. Because there are $n$ unknowns in the set of shifts, 
there is exactly one solution set $\set{a_i}_{i=1}^n$.
Fix $\sigma \neq 0$ and for all $i \in \N_n$, define $\tilde{a}_i = a_i + 
\sigma$. For each equation in $S$, 
$\tilde{a}_i - \tilde{a}_j = a_i - a_j = a_{ij}$.
Thus, $\set{\tilde{a}_i}_{i=1}^n$ is also a solution to $S$. This contradicts 
the uniqueness of $\set{a_i}_{i=1}^n$, and therefore no such set of $n$ 
linearly independent equations exists.
\end{proof}

\subsection{Discretization for the boundary 
method}\label{s:bm_math}
In the first two subsections below,
we give some results on 
how the grid-points interact with the 
underlying space. In sections \ref{s:aij_err} and \ref{s:err_c} we present 
error 
bounds.
In section \ref{s:Br_vol} we consider issues of volume and containment: 
here we ensure that one can have $B \subseteq \cl{B}^r$ for all $r$, and show 
that $\abs{\cl{B}^r} \to 0$ as $r \to \infty$.
Finally, \sref{s:Wass_err} puts bounds on the error for the \Was{} distance 
approximation.

\subsubsection{Discretization definitions}\label{s:discret}
As described in \sref{s:bm}, we discretize the region $A$ using a regular 
Cartesian grid, and refine the grid over multiple iterations, with the aim of 
refining only the grid region containing the boundary set.

\begin{definition}\label{d:adjV}
Let $\mathcal{V}$ be the set of \emph{adjacency vectors} for all discretizations 
of 
$A$.
We choose $\mathcal{V}$ to be the linear combinations of the 
standard unit vectors, $e_1,\,\ldots,\,e_d$, with coefficients 
$\pm1$.
We specifically exclude the zero vector from the set, 
so $\abs{\mathcal{V}} = 3^d -1$.
If $d = 2$, $\mathcal{V}$ equals
\begin{equation}
\mathcal{V} := \set{\, (-1,\,-1),\,(0,\,-1),\,(-1,\,0),\,(-1,\,1),\,
(1,\,-1),\,(1,\,0),\,(0,\,1),\,(1,\,1)\,
}.
\end{equation}
\end{definition}

Let $r \in \N$ be the current discretization level, and $w = w_r$ be the 
\emph{width} of the discretization at level $r$.
Let $A^r$ be the $r$-th \emph{point set}, the set of points $\vc{x}$ included 
in the $r$-th discretization of $A$.
Since we discard boxes of $\mu$-measure zero during the transport step, assume 
without loss of generality that $\mu(\vc{x}^r) > 0$ for all $\vc{x} \in A^r$.

For each iteration $r$, let
\begin{equation}\label{e:Ari}
A_i^r = A_i \cap A^r.
\end{equation}
for all $i \in \N_n$.
For all $\vc{x} \in A^r$, the points in $A^r$ that are adjacent to 
$\vc{x}$ constitute a subset of the \emph{neighbors} of $\vc{x}$,
\begin{equation}\label{e:neighb}
N(\vc{x}) := \set{\vc{x}+w_r\vc{v} \mid \vc{v} \in \mathcal{V}}.
\end{equation}

\begin{lemma}\label{l:Nsymmetric}
Let $A^r$ be the set of points included in the $r$-th discretization of $A$, and 
assume the
definition of $N$ given in \eref{e:neighb}.
For all $\vc{x},\,\vc{x}_0 \in A^r$,
if $\vc{x} \in N(\vc{x}_0)$, then $\vc{x}_0 \in N(\vc{x})$.
\end{lemma}
\begin{proof}
This follows from \eref{e:neighb} and the adjacency vectors established in 
\dref{d:adjV}:
for all $k \in \N_d$, $e_k \in \mathcal{V} \iff -e_k \in \mathcal{V}$.
\end{proof}

We now formalize our idea of the $r$-th interior and boundary point sets used 
in our discretization.
For all $i \in \N_n$, define the $r$-th iteration \emph{interior point set}
associated with $A_i$ as
\begin{equation}\label{Ai-int}
\mathring{A}_i^r := \set{\vc{x} \in A^r_i \mid
\forall \vc{v} \in \mathcal{V},\,\,
\vc{x} + w_r\vc{v} \in A^r_j \implies j = i}.
\end{equation}
Define the $r$-th \emph{boundary point set} as
\begin{equation}\label{e:Br}
B^r := A^r \setminus \bigcup_{i=1}^n \mathring{A}_i^r,
\end{equation}
and let
\begin{equation}\label{e:Bri}
B^r_i := B^r \cap A_i
\end{equation}
for all $i \in \N_n$.
The $r$-th \emph{evaluation region}, the subset of $A$ enclosed by the 
discretization $A^r$, is defined as
\begin{equation}\label{e:barAr}
\cl{A}^r := \set {\, \vc{x}^r \mid \vc{x} \in A^r \,},
\end{equation}
and the $r$-th \emph{boundary region}, the subset of $A$ enclosed by the 
boundary point set $B^r$, is given by
\begin{equation}\label{e:barBr}
\cl{B}^r := \set {\, \vc{x}^r \mid \vc{x} \in B^r \,}.
\end{equation}

\subsubsection{Distance bounds}\label{s:distance}
Though the discretization is fully defined, it still needs to be related back 
to the sets $A_{ij}$ and the boundary set $B$.
To do this, we first bound the distance separating $B^r$ and 
$A_{ij}$.

\begin{lemma}\label{l:dbnd}
Let $A^r$ be the set of points included in the $r$-th discretization of $A$, 
$w_r$ the width
at that discretization, and $\mathcal{V}$ the adjacency vector set satisfying 
\dref{d:adjV}.
Assume $A_{ij}$ is defined by \eref{Aij}, $\mathrm{edg}(\cdot)$ by 
\eref{e:edgeAr},
and $B^r_i$ by \eref{e:Bri}.
Suppose $A$ is convex, $c$ is a $p$-norm on $X\times Y$, and $B^r_i \neq 
\varnothing$.
For each $\vc{x}_i \in B^r_i$,
either $\vc{x}_i \in \edge(A^r)$ or 
there exists a point $\vc{x}_j = \vc{x}_i + 
w_r\vc{v}$, with $\vc{v} \in \mathcal{V}$, such that $\vc{x}_j \in B^r_j$ for 
some $j \neq i$.
Thus, if $\vc{x}_i \notin \edge(A^r)$, the distance from $\vc{x}_i$ to the set 
$A_{ij}$, as measured with respect to the ground cost $c$, is bounded above by 
$c(\vc{x}_i,\,\vc{x}_j)$.
\end{lemma}
\begin{proof}
Recall the definition of $A_i^r$ in \eref{e:Ari}.
Assume $\vc{x}_i \in B^r_i \setminus \edge(A^r)$. By the definition of $B^r$
given in \eref{e:Br}, 
there exists 
$\vc{x}_j = \vc{x}_i + w_r\vc{v} \in A_j^r \cup N(\vc{x}_0)$ for some $j \neq 
i$, where $N(\vc{x}_0)$ is the set of neighbors of $\vc{x}_0$ as defined in 
\eref{e:neighb}.
By \lref{l:Nsymmetric}, $\vc{x}_i \in N(\vc{x}_j)$, and since $\vc{x}_i \in 
A_i^r$, we have $\vc{x}_j \in B^r_j$.
Thus, $\vc{x}_i \in A$ and $\vc{x}_j \in A$, and because $A$ is convex, this 
implies
\begin{equation*}
\set{t\vc{x}_i + (1-t)\vc{x}_j\ | \ t \in [0,\,1]} \subseteq A.
\end{equation*}

Because $c$ is continuous on $X \times Y$, \lref{l:Fcont} applies.
Hence, $F$ is continuous on $A$.
Therefore, because $\vc{x}_i \in A_i$ and $\vc{x}_j \in A_j$, there exists
$t_* \in [0,\,1]$ such that $\vc{b} = t_*\vc{x}_i + (1-t_*)\vc{x}_j \in 
A_{ij}$.
Then $\vc{b} = \vc{x}_i + (1-t_*)w_r\vc{v}$, so by applying the ground cost we 
have
\begin{equation*}
\norm{\vc{b}-\vc{x}_i}_p = \norm{\vc(1-t_*)w_r\vc{v}}_p =
(1-t_*)\norm{w_r\vc{v}}_p \leq \norm{w_r\vc{v}}_p
= \norm{\vc{x}_j - \vc{x}_i}_p.
\end{equation*}
Therefore, $c(\vc{x}_i,\,\vc{b}) \leq c(\vc{x}_i,\,\vc{x}_j)$.
\end{proof}

Because we can bound the ground cost between the points in $B^r \setminus 
\edge(A^r)$ and the set 
$A_{ij}$ in terms of the ground cost between neighboring points, it is worth 
identifying a bound on that ground cost between neighbors.

\begin{lemma}\label{l:M_r}
Suppose $c = \norm{\cdot}_p$, $p \in [1,\,\infty]$ and assume a shift 
characterized problem in $\R^d$.
Let $N$ be defined as given by \eref{e:neighb} and $B^r_i$ as given by 
\eref{e:Bri}.
For the $r$-th iteration of the boundary method, given width $w_r$, there exists 
a maximum
$M_r$ such that, for
all $\vc{x}_i \in B^r_i$ and $\vc{x}_j \in B^r_j$, where $i,\,j \in \N_n$ 
and $i \neq j$, if $\vc{x}_j \in N(\vc{x}_i)$, then
$c(\vc{x}_i,\,\vc{x}_j) \leq M_r \leq w_rd^{1/p}$.
\end{lemma}

\begin{proof}
Let $\vc{x}_i$ and $\vc{x}_j$ be defined as above.
By applying the definition given in \eref{e:Br}, $\vc{x}_j = \vc{x}_i + 
w_r\vc{v}$ for some $\vc{v} \in \mathcal{V}$.
For our Cartesian grid $\mathcal{V}$,
$\norm{\vc{v}}_{p}$ achieves 
its maximum when $\vc{v} = \mathbbm{1}_d = (1,\,\ldots,\,1) \in \R^d$, so
\begin{equation*}
c(\vc{x}_i,\,\vc{x}_j) = \norm{w_r\vc{v}}_p
= w_r\norm{\vc{v}}_p
\leq w_r\norm{\mathbbm{1}_d}_{p}
= w_r d^{1/p}.
\end{equation*}
Therefore, there exists maximum $M_r$ such that, for all 
$\vc{x}_i \in B^r_i$ and $\vc{x}_j \in B^r_j$,
$c(\vc{x}_i,\,\vc{x}_j) \leq M_r \leq w_r d^{1/p}$.
\end{proof}

\subsubsection{Error bounds for shift 
differences}\label{s:aij_err}
In order to bound the error on the \Was{} distance, we merely require a finite 
bound on the errors for the individual shift differences, $a_{ij}$.
However, accurately computing the shift differences themselves is also 
important, and for that reason, we also present theorems that more 
finely bound the error on $a_{ij}$ for important ground cost functions.
Because estimates are generated using one or more computations of 
$g_{ij}(\vc{x})$, the magnitude of these errors is dependent on the 
point(s) chosen.

\begin{lemma}\label{l:alpha_def}
Let $A_{ij}$ be defined by \eref{Aij}, and $a_{ij}$ by \eref{e:aij}.
Suppose the ground cost $c$ satisfies the triangle inequality.
Let $\vc{x}  \in A$ and $i,\,j \in \N_n$ such that $i \neq j$.
The error resulting from approximating $a_{ij}$ at $\vc{x}$ is bounded above 
by
$\abs{\alpha_{ij}(\vc{x})} \leq 2c(\vc{x},\,\vc{b})$, where
\begin{equation}\label{e:cdif}
\alpha_{ij}(\vc{x}) :=
\big[ c(\vc{x},\,\vc{y}_i) - c(\vc{b},\,\vc{y}_i) \big]
 + \big[ c(\vc{b},\,\vc{y}_j) - c(\vc{x},\,\vc{y}_j) \big],
\end{equation}
and $\vc{b}$ is the point in $A_{ij}$ nearest to $\vc{x}$ with respect to the 
ground cost.
\end{lemma}
\begin{proof}
Assume $\vc{b} \in A_{ij}$ is the closest point in $A_{ij}$ to $\vc{x}$.
Then
\begin{equation*}
c(\vc{b},\,\vc{y}_i) -  c(\vc{b},\,\vc{y}_j) = g_{ij}(\vc{b}) = a_{ij}.
\end{equation*}
For every $\vc{x} \in A$, there exists some $\alpha_{ij}(\vc{x}) \in \R$ such 
that
\begin{equation*}
c(\vc{x},\,\vc{y}_i) -  c(\vc{x},\,\vc{y}_j) = a_{ij} + \alpha_{ij}.
\end{equation*}
By rearrangement and substitution, we have
\begin{align*}
\alpha_{ij}(\vc{x}) &= -a_{ij} + c(\vc{x},\,\vc{y}_i) -  c(\vc{x},\,\vc{y}_j) \\
&= -\big[ c(\vc{b},\,\vc{y}_i) -  c(\vc{b},\,\vc{y}_j) \big]
+ c(\vc{x},\,\vc{y}_i) -  c(\vc{x},\,\vc{y}_j) \\
&= \big[ c(\vc{x},\,\vc{y}_i) - c(\vc{b},\,\vc{y}_i) \big]
+ \big[ c(\vc{b},\,\vc{y}_j) - c(\vc{x},\,\vc{y}_j) \big].
\end{align*}
Since $c$ satisfies the triangle inequality,
\begin{equation*}
c(\vc{x},\,\vc{y}_i) - c(\vc{b},\,\vc{y}_i)
\leq c(\vc{x},\,\vc{b}) + c(\vc{b},\,\vc{y}_i) - c(\vc{b},\,\vc{y}_i)
= c(\vc{x},\,\vc{b}).
\end{equation*}
Thus,
\begin{equation*}
\abs{ c(\vc{x},\,\vc{y}_i) - c(\vc{b},\,\vc{y}_i) }
\leq \abs{ c(\vc{x},\,\vc{b}) } = c(\vc{x},\,\vc{b}),
\end{equation*}
and, by a similar line of reasoning,
$\abs{ c(\vc{b},\,\vc{y}_j) - c(\vc{x},\,\vc{y}_j) }
\leq c(\vc{x},\,\vc{b})$.
Therefore,
\begin{equation*}
\abs{\alpha_{ij}(\vc{x})} \leq
\abs{ c(\vc{x},\,\vc{y}_i) - c(\vc{b},\,\vc{y}_i) }
+ \abs{ c(\vc{b},\,\vc{y}_j) - c(\vc{x},\,\vc{y}_j) } \leq 2c(\vc{x},\,\vc{b}).
\end{equation*}
\end{proof}

In addition to bounding the error for individual points $\vc{x}$, we can also 
establish meaningful global bounds.

\begin{lemma}\label{l:tri_2del}
Assume a shift characterized problem in $\R^d$, where $c$ is a $p$-norm, $p \in 
[1,\,\infty]$.
Let $w_r$ be the width of the discretization during iteration $r$, and let 
$\vc{x}^r$ indicate the box of
width $w_r$ centered at the point $\vc{x}$.
Let $B$ be defined by \eref{bdryset}, $N$ by \eref{e:neighb}, and $B^r_i$ by 
\eref{e:Bri}.
Taking $\alpha_{ij}$ as defined by \eref{e:cdif} and $\cl{B}^r$ as given by 
\eref{e:barBr},
let $\alpha_{\max}$ be the maximum value of 
$\abs{\alpha_{ij}(\vc{x})}$ over all $\vc{x} \in \cl{B}^r$ and $i,\,j \in 
\N_n$, such that:
(1) $i \neq j$,
(2) $\vc{x} \in \vc{x}_i^r$ for some $\vc{x}_i \in B^r_i$, and
(3) $B^r_j \cap N(\vc{x}_i) \neq \varnothing$.
Then
$\alpha_{\max} \leq 4w_rd^{1/p}$
and
for all $\vc{x} \in \cl{B}^r$,
$\norm{\vc{x}-B}_p \leq 2w_rd^{1/p}$.
\end{lemma}
\begin{proof}
Suppose $\vc{x} \in \cl{B}^r$.
By the definition of our grid, $\vc{x}$ is contained in some $G = \Conv(S)$, 
where $S$ is a finite set of neighboring grid points.
For each $\vc{x}_a,\,\vc{x}_b \in S$, $\vc{x}_b \in N(\vc{x}_a)$, and hence
$\vc{x}_b = \vc{x}_a + w_r\vc{v}$ for some $\vc{v} \in \mathcal{V}$,
the adjacency vectors described in \dref{d:adjV}.
Since $\norm{v}_p \leq d^{1/p}$,
$c(\vc{x}_a,\,\vc{x}_b) \leq w_rd^{1/p}$.
Because $\vc{x}_a$ and $\vc{x}_b$ were arbitrarily chosen, this is true of 
every pair of vertices of $G$.
By the definition of $G$, $\vc{x}$ can be written as a convex combination of 
the points in $S$.
Therefore, for any fixed $\vc{x}_0 \in S$, $c(\vc{x},\,\vc{x}_0) \leq 
w_rd^{1/p}$.

Recall the definition of $B^r$ given in \eref{e:Br}.
Because $\vc{x} \in \cl{B}^r$, $\Conv(S) \cap B^r$ must be nonempty.
Assume without loss of generality that $\vc{x}_0 = \vc{x}_i \in B^r_i$ for some 
$i \in \N_n$.
Because $c$ satisfies the triangle inequality, \lref{l:alpha_def} applies.
Hence, there must exist a point $\vc{x}_j \in B^r_j$, a 
neighbor of $\vc{x}_i$, with $j \neq i$, and a point $\vc{b} \in A_{ij}$ such 
that
$c(\vc{x}_i,\,\vc{b}) \leq c(\vc{x}_i,\,\vc{x}_j) \leq w_rd^{1/p}$.
Applying the triangle inequality, we find that
\begin{equation*}
c(\vc{x},\,\vc{b}) \leq
c(\vc{x},\,\vc{x}_i) + c(\vc{x}_i,\,\vc{b})
\leq 2w_rd^{1/p}.
\end{equation*}
Therefore,
$\norm{\vc{x}-B}_p \leq 2w_rd^{1/p}$
and
$\alpha_{\max} \leq 4w_rd^{1/p}$.
\end{proof}

\subsubsection{Error bound for ground costs}\label{s:err_c}
In preparation for bounding the \Was{} distance error, we now bound the error 
on the ground cost $c$ with respect to individual points in $\cl{B}^r$.

\begin{lemma}\label{l:gamma_max}
Given a shift characterized transport problem in $\R^d$, with ground cost $c = 
\norm{\cdot}_p$, $p \in [1,\,\infty]$.
Assume $w_r$ is the width of the discretization at the $r$-th iteration and
let $\tilde{\pi}^*$ be an approximated transport plan with associated transport 
map 
$\widetilde{T}$, obtained using the boundary method with discretization $w_r$.
Suppose $\pi^*$ is an optimal transport plan with associated map $T$, and let 
$\vc{x}$ in $A$ such that $T(\vc{x}) = \vc{y}_i$, but 
$\widetilde{T}(\vc{x}) = \vc{y}_j$.
Then the error in the ground cost at the point $\vc{x}$ is equal to 
$\abs{g_{ij}(\vc{x})}$,
where $g_{ij}$ is defined as given in \eref{e:star}.
Furthermore, there exists $\gamma_{\max}$ such that, for all such $\vc{x} \in A$ 
with
$T(\vc{x}) = \vc{y}_i$ and $\widetilde{T}(\vc{x}) = \vc{y}_j$ for some $i \neq 
j$,
\begin{equation}\label{e:gamma_max}
\abs{g_{ij}(\vc{x})} \leq \gamma_{\max} \leq
\max_{\substack{1 \leq i < n\\i < j \leq n\\A_{ij} \neq 
\varnothing}}\,\abs{a_{i}-a_j}
+ 4w_rd^{1/p} < \infty.
\end{equation}
\end{lemma}
\begin{proof}
Let $\vc{x} \in A$ such that $T(\vc{x}) = \vc{y}_i$, but 
$\widetilde{T}(\vc{x}) = \vc{y}_j$.
Then the error in the ground cost at $\vc{x}$ equals
\begin{equation*}
\abs{ c(\vc{x},\,\vc{y}_i) - c(\vc{x},\,\vc{y}_j) }
= \abs{ g_{ij}(\vc{x}) }.
\end{equation*}
As a consequence of \lref{l:tri_2del}:
\begin{equation*}
\abs{g_{ij}(\vc{x})} =
\abs{c(\vc{x},\,\vc{y}_i) -  c(\vc{x},\,\vc{y}_j)}
\leq \abs{a_{ij}} + \abs{\alpha_{ij}(\vc{x})}
\leq 
\max_{\substack{1 \leq i < n\\i < j \leq n}}\,\abs{a_{ij}}
+ 4w_rd^{1/p} < \infty.
\end{equation*}
The result is independent of $\vc{x}$, $i$, and $j$,
and therefore
there must exist some
$\gamma_{\max} \leq \max_{\substack{1 \leq i < n\\i < j \leq n}}\,\abs{a_{ij}}
+ 4w_rd^{1/p} < \infty$.
\end{proof}

\subsubsection{Volume and containment for the boundary 
region}\label{s:Br_vol}
As shown in \sref{s:err_c}, the ground cost error for individual points is 
finitely bounded over a wide range of admissible ground cost functions.
By definition, the measure $\mu$ is bounded.
We propose to identify the largest possible region in which 
the ground cost error can be non-zero, and to show that the area of that region 
goes to zero as $r$ goes to infinity.  With this,
we will show that the boundary method 
converges with respect to the \Was{} distance.

In \eref{e:barBr}, we defined a region $\cl{B}^r$ based on the point set $B^r$.
For this, we need to know that we can choose an initial width $w_1$ such 
that, for all iterations $r$, $B \subset \cl{B}^r$.
\tref{t:goodw1} guarantees that such a width exists, and gives a sense of
the relevant features driving the choice of $w_1$.
For details about the numerical considerations involved, see
\rref{r:goodw1}.

\begin{theorem}\label{t:goodw1}
Assume $c$ is a $p$-norm, $p \in [1,\,\infty]$, and $A = [0,\,l]^d$.
There exists an initial width $w_1$ such that, for all 
$w_r$ such that $w_r \leq w_1$, $\vc{x} \in \mathring{A}_i^r$,
as defined by \eref{Ai-int}, implies 
the box of width $w_r$ centered at $\vc{x}$, given by $\vc{x}^r$,
satisfies $\vc{x}^r \subseteq \mathring{A}_i$, where $\mathring{A}_i$ is the 
strict 
interior of $A_i$, as defined by \eref{e:oAi}.
\end{theorem}
\begin{proof}
Recall the definition of $A_i$ given by \eref{Aidef},
$A_{ij}$ given by \eref{Aij}, $B$ given by \eref{bdryset},
and $g_{ij}$ given by \eref{e:star}.
Let
$\mathcal{B}(\vc{x},\,s)$ indicate the open ball of radius $s$ (with respect to 
the $p$-norm $c$) centered at $\vc{x}$ and
$\mathcal{C}(\vc{x},\,s)$ indicate the $d$-dimensional cube with side
length $s$ (with respect to the Euclidean distance) centered at $\vc{x}$.
Because $c$ is a $p$-norm, for each $i \in \N_n$,
$\vc{y}_i \in \mathring{A}_i$, and therefore there exists
$\delta_i > 0$ such that
$\mathcal{B}(\vc{y}_i,\,\delta_i) \subseteq \mathring{A}_i$.
Thus, there exist $\varepsilon > 0$ and $\delta \geq \varepsilon$, such that,
for any $i \in \N_n$,
$c(\vc{x},\,\vc{y}_i) < \delta$ implies
$\mathcal{C}(\vc{x},\,4s) \subseteq \mathring{A}_i$ for all
$s \leq \varepsilon$.

Let
\begin{equation*}
S := A \setminus \bigcup_{i \in \N_n}\, \mathcal{B}(\vc{y}_i,\,\delta).
\end{equation*}
Because $S$ is a closed set minus a finite number of open sets, $S$ is closed.

If all $A_{ij}$ are hyperplanes on $S$, the claim is self-evident
for all $w_1 \leq \varepsilon$, so assume instead that 
at least one $A_{ij}$ is not a hyperplane on $S$.
Let
\begin{equation*}
G_1 := \sup_{\substack{\vc{x} \in S\\i,\,j \in \N_n\\i \neq j}}
\, \abs{\nabla g_{ij}(\vc{x})}
\quad\quad\text{ and }\quad\quad
G_2 := \sup_{\substack{\vc{x} \in S\\i,\,j \in \N_n\\i \neq j}}
\, \abs{\nabla^2 g_{ij}(\vc{x})}.
\end{equation*}
There exists a maximum directional magnitude with respect to the
Euclidean distance,
\begin{equation*}
M = \sup_{\substack{\vc{x} \in S\\\vc{u} \in \R^d\\i \in \N_n}}
\, \abs{x_\vc{u}-y^i_\vc{u}} \leq l\sqrt{d}<\infty,
\end{equation*}
where $\abs{x_\vc{u}-y^i_\vc{u}}$ is the magnitude of the vector
$\vc{x} - \vc{y}_i$ projected parallel to the direction of $\vc{u}$.
Because $c \in C^2(S)$, $G_1$ and $G_2$ are well-defined.
For any $\vc{x} \in S$ and any unit direction vector $\vc{u} \in \R^d$,
\begin{align*}
\abs{\nabla_\vc{u}\,g_{ij}(\vc{x})}
&= \abs*{
\frac{(x_{\vc{u}}-y^i_{\vc{u}})\abs{x_{\vc{u}}-y^i_{\vc{u}}}^{p-2}}
{(c(\vc{x},\,\vc{y}_i))^{p-1}}
-
\frac{(x_{\vc{u}}-y^j_{\vc{u}})\abs{x_{\vc{u}}-y^j_{\vc{u}}}^{p-2}}
{(c(\vc{x},\,\vc{y}_j))^{p-1}}
}
\\
&\leq 2\frac{M^{p-1}}{\delta^{p-1}} < \infty.
\end{align*}
and
\begin{align*}
\abs{\nabla^2_\vc{u}\,g_{ij}(\vc{x})}
&=
(p-1)\left|
\frac{\abs{x_{\vc{u}}-y^i_{\vc{u}}}^{p-2}}
{(c(\vc{x},\,\vc{y}_i))^{p-1}}
-
\frac{\abs{x_{\vc{u}}-y^i_{\vc{u}}}^{2p-2}}
{(c(\vc{x},\,\vc{y}_i))^{2p-1}}
\right.
\\
&\phantom{=(p-1)|}\quad\quad
\left.
-
\frac{\abs{x_{\vc{u}}-y^j_{\vc{u}}}^{p-2}}
{(c(\vc{x},\,\vc{y}_j))^{p-1}}
+
\frac{\abs{x_{\vc{u}}-y^j_{\vc{u}}}^{2p-2}}
{(c(\vc{x},\,\vc{y}_j))^{2p-1}}
\right|
\\
&\leq
2(p-1)\frac{M^{p-2}}{\delta^{p-1}}\left[
1 + \frac{M^p}{\delta^p}
\right] < \infty.
\end{align*}
Hence,
$G_1 < \infty$
and $G_2 < \infty$.

Assume the Gaussian curvature of the set $A_{ij}$ at a point
$\vc{x} \in A_{ij}$ is given by the function $K_{ij}(\vc{x})$,
and when $K_{ij}(\vc{x})\neq 0$
the radius of curvature is given by
$R_{ij}(\vc{x}) = \abs{K_{ij}(\vc{x})}^{-1}$.
Because $K_{ij}(\vc{x})$ is defined as a product of first and second
directional derivatives of $g_{ij}$, and those derivatives are
bounded, there exists a maximum absolute Gaussian curvature for
$B$ on $S$, given by
\begin{equation*}
K := \sup_{\substack{i,\,j \in \N_n\\i\neq j\\\vc{x} \in A_{ij} \cap S}}
\,\abs{K_{ij}(\vc{x})} < \infty.
\end{equation*}
Because at least one $A_{ij}$ is not a hyperplane, $K > 0$.
Because $K < \infty$, for any $i,\,j \in \N_n$, $i \neq j$ and
any $\vc{x} \in A_{ij} \cap S$, the radius of curvature is
bounded below: $R_{ij}(\vc{x}) \geq K^{-1} > 0$.

Let $\tilde{\varepsilon} = \dfrac{2}{K\sqrt{d}}$.
Suppose $s \leq \min\set{\varepsilon,\,\tilde{\varepsilon}}$, $\vc{x}_0 \in 
\mathring{A}_i^r$
for some $i \in \N_n$, and
that
$\mu(A_j \cap \mathcal{C}(\vc{x}_0,\,s)) > 0$ for some
$j \in \N_n$, $j \neq i$.

The set $\mathcal{C}(\vc{x}_0,\,2s)$ is the cube surrounding $\vc{x}_0$
and its neighbors.
Because $\vc{x}_0 \in \mathring{A}_i^r$, $A_{ij}$
cannot be a hyperplane in $\mathcal{C}(\vc{x}_0,\,2s)$, and so $R_{ij}$
is well-defined on $\mathcal{C}(\vc{x}_0,\,2s)$.
If there exist $k \in \N_n$ and
$\vc{x} \in \mathcal{C}(\vc{x}_0,\,2s)$ such that
$c(\vc{x},\,\vc{y}_k) < \delta$, then
$\mathcal{C}(\vc{x}_0,\,2s)
\subseteq \mathcal{C}(\vc{x},\,4\varepsilon) \subseteq \mathring{A}_k$,
and since $\vc{x}_0 \in \mathring{A}_i^r$, this implies $k = i$ and
$\mathcal{C}(\vc{x}_0,\,2s) \subseteq \mathring{A}_i$.
This implies $\mathcal{C}(\vc{x}_0,\,s) \cap A_{j} = \varnothing$,
which contradicts the claim that 
$\mu(A_j \cap \mathcal{C}(\vc{x}_0,\,s)) > 0$.
Therefore, 
$c(\vc{x},\,\vc{y}_k) \geq \delta$ for all $k \in \N_n$
and $\vc{x} \in \mathcal{C}(\vc{x}_0,2s)$.
This implies $\mathcal{C}(\vc{x}_0,\,2s) \subseteq S$.
Hence, the intersection of the boundary $A_{ij}$ with the cube
$\mathcal{C}(\vc{x}_0,\,2s)$
must have a point with minimum radius of curvature,
\begin{equation*}
\vc{x}_m := \arginf_{\vc{x} \in A_{ij} \cap \mathcal{C}(\vc{x}_0,\,2s)}
\,R_{ij}(\vc{x}),
\end{equation*}
and since $\vc{x}_m \in S$, it must be the case that
$R_{ij}(\vc{x}_m) \geq \sfrac{1}{K}$.

Because $\mu(A_j \cap \mathcal{C}(\vc{x}_0,\,s)) > 0$, but
$\vc{x}_0 \notin A_j$,
there must exist
$\vc{x}_c \in A_{ij} \cap \mathcal{C}(\vc{x}_0,\,s)$.
Hence, within the cube $\mathcal{C}(\vc{x}_0,\,2s)$,
there must be a $d$-dimensional
sphere (or partial sphere) of radius $R_{ij}(\vc{x}_m)$,
not in $A_i$,
whose boundary intersects $\vc{x}_c$
(``partial'' because the sphere may be cut off
by one or more of the planes bounding the cube).
Call this (partial) sphere $\widetilde{\mathcal{S}}$.

Since $\vc{x}_0 \in \mathring{A}_i^r$, it must be the case
that $\widetilde{\mathcal{S}} \cap \set{N(\vc{x}_0) \cup \set{\vc{x}_0}}
= \varnothing$, where $N$ is the set of neighbors defined by \eref{e:neighb}.
Because $\vc{x}_c \in \mathcal{C}(\vc{x}_0,\,s)$,
and the maximum distance between grid points in
$\mathcal{C}(\vc{x}_0,\,2s)$ is $s\sqrt{d}$,
this requires $R_{ij}(\vc{x}_m) < s\sqrt{d}/2$.
Hence, there exists $\vc{x}_m \in A_{ij} \cap S$ such that
\begin{equation*}
R_{ij}(\vc{x}_m) < \frac{s\sqrt{d}}{2} \leq \frac{1}{K}.
\end{equation*}
This contradicts $R_{ij}(\vc{x}_m) \geq \sfrac{1}{K}$.
Thus, it must be the case that for all $j \in \N_n$, $j \neq i$ implies
$\mu(\mathcal{C}(\vc{x},s) \cap A_{j}) = 0$, and therefore
$C(\vc{x},\,s) \subseteq \mathring{A}_i$.

Setting $w_1 \leq \min\set{\varepsilon,\,\tilde{\varepsilon}}$
completes the proof.
\end{proof}

\begin{rem}\label{r:goodw1}
By the boundary method's very nature, any initial configuration must enclose the 
boundary
in a way that allows it to be distinguished from the region interiors.
This is the meaning behind the width $w_1$ considered in \tref{t:goodw1}.
In principle, $w_1$ may need to be quite small.
In practice, the potential problems associated with an overly-large $w_1$ 
rarely 
occur, and they are obvious when they do.
We did occasionally observe an issue when the initial $w_1$ was so large that a 
region $\vc{x}^r$ could contain an entire $A_i$
(in other words, when $w_1$ was significantly
larger than the $\delta$ described in
\tref{t:goodw1}).
In those cases, the affected region's $a_i$ was such that
$c(\vc{y}_i,\,\vc{y}_j) = a_i - a_j$ for some $j \neq i$, and
the resulting transport plan had $\mu(A_i) = 0$.
Hence, the set $\set{a_i}_{i=1}^n$ and reconstructed regions $\set{A_i}_{i=1}^n$ 
directly revealed
when such an error had occurred.

Also, because of the nature of the iterative method, a poor choice of $w_1$
quickly becomes obvious in the boundary region
itself.
Simply put, the loss of any portion of the boundary set $B$ destabilizes
the method.
Losing part of $B$ creates a visible gap in the ``wall'' between two
regions, and the gap increases in size with each successive iteration.
This behavior seems to occur whenever some part of $B$ is lost, no matter
what the cause.
For example, in our tests we observed that discarding an edge box that 
intersects $B$ results in the same progressive damage to the boundary set.
Not surprisingly, this also ``stalls'' the convergence of the Wasserstein
distance in ways that are obvious during computation.

In our numerical tests, we used $w_1 \leq \sfrac{1}{50n}$ and obtained
consistently reliable results.
\end{rem}

Next, we show that a well-chosen initial width and grid 
arrangement can guarantee that, for every iteration $r$, each point in
$A^r \setminus B^r$ corresponds to a box in the interior of some
region $A_i$.

\begin{theorem}\label{t:BinBr}
Assume $c$ is a $p$-norm, $p \in [1,\,\infty]$, and $A = [0,\,l]^d$.
Suppose the first iteration width $w_1$ is chosen as described in 
\tref{t:goodw1}.
Fix $r$, let $w_r \leq w_1$, and let $A^r$ be the boundary set remaining at the
$r$-th iteration.
Given the definition of $B$ from \eref{bdryset}, $\cl{A}^r$ from \eref{e:barAr},
$\cl{B}^r$ from \eref{e:barBr},
if $B \subseteq \cl{A}^r$, then $B \subseteq \cl{B}^r$, and hence $B \subseteq 
A^{r+1}$.
\end{theorem}
\begin{proof}
We will show the conclusions by proving that $\vc{x}_0 \notin \cl{B}^r$ implies 
$\vc{x}_0 \notin B$.

Suppose $\vc{x}_0 \notin \cl{B}^r$.
If $\vc{x}_0 \notin \cl{A}^r$, then $\vc{x}_0 \notin B$, since by 
assumption, $B \subseteq \cl{A}^r$.
Thus, we assume instead that $\vc{x}_0 \in \cl{A}^r \setminus \cl{B}^r$.

Because $\vc{x}_0 \in \cl{A}^r$, we know $\vc{x}_0 \in \vc{x}^r$, the box of
radius $w_r$ centered around some $\vc{x} \in A^r$.
We have $\vc{x} \in A_i$ for some $i \in \N_n$, where $A_i$ is defined as given 
in
\eref{Aidef}, and so by the definition of $A_i^r$ from \eref{e:Ari}, $\vc{x} \in 
A_i^r$.
However, $\vc{x}_0 \notin \cl{B}^r$ implies $\vc{x}^r \not\subseteq \cl{B}^r$, 
so from the definition of $B^r$ given in \eref{e:Br}, $\vc{x} \notin B^r$.
Because, $\vc{x} \in A_i^r \setminus B^r = \mathring{A}_i^r$ (see
\eref{Ai-int}), by 
\tref{t:goodw1}, $\vc{x}^r \subseteq \mathring{A}_i$. Hence, 
$\vc{x}_0 \in \mathring{A}_i$. Therefore, by \eref{e:oAi}, $\vc{x}_0 
\notin B$.
\end{proof}

Now that we have ensured $B \subseteq \cl{B}^r$, we aim to construct a region
of controlled volume enclosing $\cl{B}^r$: $\cl{B}^r \subseteq 
\clBrp$.
Then we show that, as $r \to \infty$, the volume 
of $\clBrp$ in $\R^d$ goes to zero with respect to the Lebesgue measure.
This will allow us to put a convenient upper bound on the volume of $\cl{B}^r$ 
in terms of the width $w_r$.
Because $\clBrp$ exists solely in $A$, and not on the product space, we can 
once again rely on the Euclidean distance in $\R^d$.

\begin{lemma}\label{l:subsetB_r}
Assume a shift characterized transport problem in $\R^d$, with 
$c=\norm{\cdot}_p$, $p \in [1,\,\infty]$.
Suppose $w_r$ is the width used for the $r$-th iteration, and assume $B$ is 
defined as given in \eref{bdryset},
$\cl{B}^r$ as given in \eref{e:barBr}.
Let the region $\clBrp \subseteq A$ be defined as
\begin{equation}\label{e:barBrp}
\clBrp := \set{ \vc{x} \in A \mid \norm{\vc{x}-B}_2
 \leq 2w_r\sqrt{d} }.
\end{equation}
For all $r$, $\cl{B}^r \subseteq \clBrp$.
\end{lemma}
\begin{proof}
By definition, $\cl{B}^r \subseteq A$.
Suppose $\vc{x} \in \cl{B}^r$.
Because we are applying the Euclidean norm, \lref{l:tri_2del} implies
that $\norm{\vc{x}-B}_2 \leq 2w_r\sqrt{d}$, and since $\vc{x} \in A$,
$\vc{x} \in \clBrp$.
\end{proof}

\begin{theorem}\label{t:clBprVolBnd}
Assume $c$ is a $p$-norm, $p \in [1,\,\infty]$.
Let $w_r$ be the width of the discretization applied during the $r$-th 
iteration.
Given the definition of $B$ in \eref{bdryset} and $\cl{B}^r$ in \eref{e:barBr},
if $\mu(B)=0$ and
there exists some constant $\tilde{L}$ such that $\abs{B} 
= \tilde{L} < \infty$ with respect to the $\R^{d-1}$ Lebesgue 
measure, then there exists some $L < \infty$, such that
$\abs*{\cl{B}^r} \leq w_r^dL$ with respect to the $\R^d$ Lebesgue 
measure.
\end{theorem}
\begin{proof}
Recall the definition of $\clBrp$ given in \eref{e:barBrp}.
We know
$ \int_{\clBrp{}} \,d\vc{x} = \int_A \chi\left[\clBrp{}\right]
\!\!(\vc{x})\,d\vc{x}$.
Let $\mathcal{B}(\vc{x},\,\rho)$ be the closed ball of radius
$\rho$ centered at $\vc{x}$, and defined with respect to the Euclidean distance.
Write
\begin{align*}
\int_A \chi\left[\clBrp{}\right]\!\!(\vc{x}) \,d\vc{x}
&= \int_A \chi
\left[\set*{\vc{x} \in A \mid \norm{\vc{x}-B}_2 \leq 2w_r\sqrt{d}}\right]
\!\!(\vc{x}) \,d\vc{x}
\\ &= \int_A \chi
\left[\set*{\vc{x} \in A \mid \vc{x} \in \mathcal{B}(\vc{z},\,2w_r\sqrt{d})
\text{ for some }
\vc{z} \in B}\right]
\!\!(\vc{x}) \,d\vc{x},
\\ &\leq \int_A \chi[B](\vc{z})\, \left( \int_A
\chi
\left[\mathcal{B}(\vc{z},\,2w_r\sqrt{d})\right]
\!\!(\vc{x}) \,d\vc{x}\right) \,d\vc{z}.
\end{align*}

For all fixed $\vc{x}$,
\begin{align*}
\int_A
\chi
\left[\mathcal{B}(\vc{x},\,2w_r\sqrt{d})\right]
\!\!(\vc{x}) \,d\vc{x}
\leq \mathrm{Vol}_d \left(2w_r\sqrt{d}\right),
\end{align*}
where $\mathrm{Vol}_d(\rho)$ is the volume of the $d$-dimensional sphere of 
radius
$\rho$, defined with respect to the Euclidean distance.
By using the $\mathrm{\Gamma}$ function, this volume can be written as
\begin{align*}
\mathrm{Vol}_d \left(2w_r\sqrt{d}\right) :=
&\frac{\pi^{d/2}}{\frac{d}{2}\mathrm{\Gamma}\left(\frac{d}{2}\right)}
\left(2w_r\sqrt{d}\right)^d
\\ =
&\begingroup
\begin{cases}
\frac{\pi^{d/2}}{(d/2)!}\left(2w_r\sqrt{d}\right)^d
& \text{ if } d = 2k \text{ for some } k \in \Z
\\
\frac{2k!(4\pi)^{k}}{d!}\left(2w_r\sqrt{d}\right)^d
& \text{ if } d = 2k+1 \text{ for some } k \in \Z .
\end{cases}
\endgroup
\end{align*}
Because the volume is independent of the point $\vc{x} \in A$, we therefore have
\begin{align*}
\int_{\clBrp{}} \,d\vc{x}
&\leq \int_A \chi[B](\vc{z})\, \int_A
\chi
\left[\mathcal{B}(\vc{z},\,2w_r\sqrt{d})\right]
\!\!(\vc{x}) \,d\vc{x} \,d\vc{z},
\\&\leq
\int_A \chi[B](\vc{z})\, \mathrm{Vol}_d \left(2w_r\sqrt{d}\right) \,d\vc{z}
= \mathrm{Vol}_d \left(2w_r\sqrt{d}\right) \int_B \,d\vc{z}
= w_r^dL,
\end{align*}
where
\begin{equation*}
L :=
\begingroup
\begin{cases}
\tilde{L}
\frac{\pi^{d/2}}{(d/2)!}
\left(2\sqrt{d}\right)^d
& \text{ if } d = 2k \text{ for some } k \in \Z
\\
\tilde{L}
\frac{2k!(4\pi)^{k}}{d!}
\left(2\sqrt{d}\right)^d
& \text{ if } d = 2k+1 \text{ for some } k \in \Z.
\end{cases}
\endgroup
\end{equation*}
Let $\vc{x} \in \cl{B}^r$. By applying
\lref{l:tri_2del} with $c$ the Euclidean distance, we know that for all $\vc{x} 
\in \cl{B}^r$,
$\norm{\vc{x}-B}_2 \leq 2w_r\sqrt{d}$, which implies
$\vc{x} \in \clBrp{}$.
Thus, $\cl{B}^r \subseteq \clBrp{}$, which implies
$\abs{\cl{B}^r} \leq \abs{\clBrp{}} \leq w_r^dL$.
\end{proof}

\begin{rem}
The interplay between $B$, $B^r$, $\cl{B}^r$, and $\clBrp$ is nontrivial. 
\fref{f:theBrs} helps to visualize it properly.
In \fsubref{f:theBrs}{f:Br}, we show placement of some boundary set $B^r$.
It is crucial that the subgrid created by $B^r$ completely surrounds $B$, 
because that is the only way to ensure that $B \subseteq \cl{B}^r$.
One can see in this image how a (very degenerate) choice of $c$, coupled with 
the right arrangement of $\vc{y}_i$'s, might allow a small and sharply curved 
boundary set to slip unnoticed between points.

As \fsubref{f:theBrs}{f:barBr} illustrates, each point in $B^r$ appears as the 
center of its corresponding box, and the boxes completely cover the boundary 
set.

The region $\clBrp$ is deliberately constructed to entirely cover all the boxes 
in $\cl{B}^r$. As \fsubref{f:theBrs}{f:barBrPLUS} shows, its volume can be 
significantly 
larger than that of the boxes it contains. However, the worst-case 
``thickness'' given 
to $\clBrp$ ensures that it will always enclose both $B$ and $\cl{B}^r$.

\begin{figure}[htpb]
  \centering
  \subfloat[$B^r$ surrounding $B$]{%
    \label{f:Br}
    \centering
    \resizebox{0.3\textwidth}{!}{%
    \begin{overpic}[width=\textwidth]{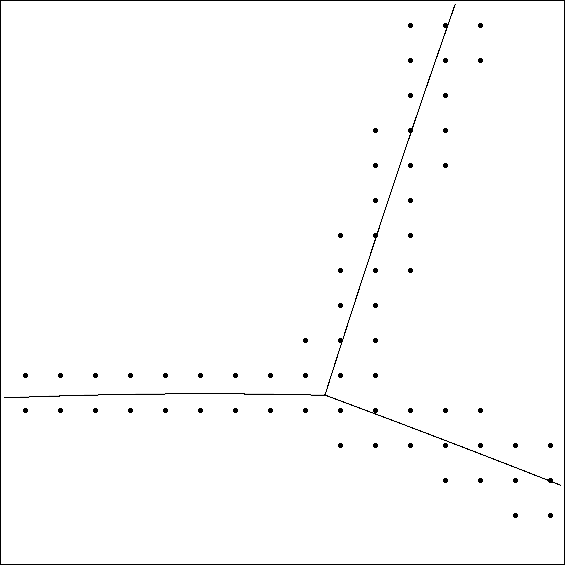}
    \end{overpic}}
  }%
  \subfloat[boxes $\cl{B}^r$ covering $B$]{%
    \label{f:barBr}
    \centering
    \resizebox{0.3\textwidth}{!}{%
    \begin{overpic}[width=\textwidth]{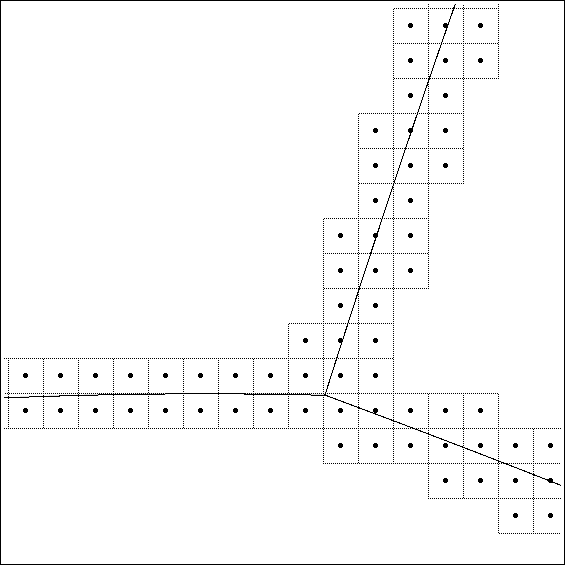}
    \end{overpic}}
  }%
  \subfloat[region $\clBrp$ covering $\cl{B}^r$]{%
    \label{f:barBrPLUS}
    \centering
    \resizebox{0.3\textwidth}{!}{%
    \begin{overpic}[width=\textwidth]{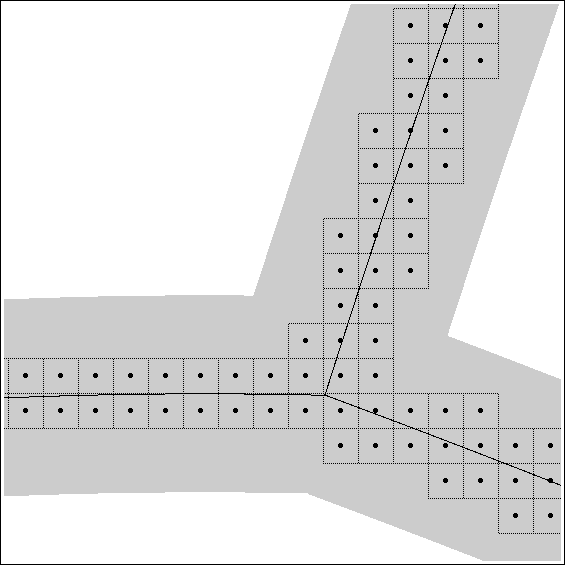}
    \end{overpic}}
  }%
\caption{Detail from problem in \fsubref{f:meas_comp}{f:uni}: Boundary 
set interactions near $A_0 \cap A_2 \cap A_3$}\label{f:theBrs}
\end{figure}
\end{rem}

\subsubsection{The \Was{} distance error}\label{s:Wass_err}
\begin{theorem}\label{t:WassErr}
Assume $\mu$ is absolutely continuous and let $P^*$ be the \Was{} distance.
Let $w_r$ be the width of the $r$-th iteration of the boundary method.
Given the definition of $B$ in \eref{bdryset} and $\cl{B}^r$ in \eref{e:barBr},
suppose $B \subseteq \cl{B}^r$, 
and that there exists some $L$ such that $\abs{\cl{B}^r} = w_r^dL < \infty$
with respect to the $d$-dimensional Lebesgue measure.
If $\gamma_{\max} < \infty$ is the maximum error of the ground cost in the set
$\cl{B}^r$, and $\widetilde{P}^*$ is the \Was{} distance approximation
obtained with the boundary method, then 
the value of $\mu$ on $A$ is bounded by some $M < \infty$ and
\begin{equation}
\abs*{\widetilde{P}^{*} - P^*} \leq
w_r^dLM\gamma_{\max} ,
\end{equation}
where the bound equals
the maximum possible volume of $\cl{B}^r$ multiplied by the maximum value of 
$\mu$ and the maximum error of the ground cost.
\end{theorem}
\begin{proof}
If $\vc{x} \in A \setminus \cl{B}^r$, then $\vc{x}$ has been identified as 
being in the interior of $A_i$ for some $i \in \N_n$.
Thus, the cost error associated with the points outside $\cl{B}^r$ is zero.

Suppose instead that $\vc{x} \in \cl{B}^r$.
By definition, the absolute value of the difference between the correct and 
approximated ground costs at $\vc{x}$ is less than or equal to $\gamma_{\max}$.
\cref{c:absCont} requires $\mu$ to be absolutely continuous, so there exists
$M$ such that, for all $\vc{x} \in X$, $0\leq \mu(\vc{x}) \leq M < \infty$.

Therefore, the error on the \Was{} distance is bounded above by
\begin{equation*}\begin{split}
\abs*{\widetilde{P}^{*} - P^*}
& \leq \int_{\cl{B}^r} \gamma_{\max} \, d\mu (\vc{x})
= \int_{\cl{B}^r} \mu(\vc{x})\gamma_{\max} \, d\vc{x} \\
& \leq \int_{\cl{B}^r} (M)\gamma_{\max} \, d\vc{x}
= \abs*{\cl{B}^r}M\gamma_{\max}
\leq (w_r^dL)M\gamma_{\max}.
\end{split}\end{equation*}
\end{proof}

\begin{rem}\label{RemWassErr}
The bounds in Theorems~\ref{t:clBprVolBnd} and \ref{t:WassErr} indicate that the
volume of the boundary set and the error of the computed \Was{} distance
decrease according to the dimension of the space.
Thus, we should expect our numerical tests to show a quadratic (in $\R^2$)
or cubic (in $\R^3$) decrease of the \Was{} distance error.
These decreases are clearly observed in practice, see \sref{s:results}.
\end{rem}

\section{Numerical results}\label{s:results}
\subsection{Test conditions}\label{s:numerics}
As mentioned in \rref{r:approx}, some choices of $\mu$ may make it 
necessary to approximate $\mu(\vc{x}^r)$ or $\displaystyle{\int_{\vc{x}^r} 
c(\vc{z},\,\vc{y}_i) \,d\mu(\vc{z})}$.
However, the majority of numerical studies we have seen restrict to
simple choices of $\mu$ (most often uniform).  For this reason, we
restricted our examples to cases where the cost and mass integrals can be 
written in a closed form.

\subsubsection{The closed-form mass $\mu(\vc{x}^r)$}\label{s:comp_mass}
The integral of $\mu$ over some box can be written as:
\begin{equation}
\mu(\vc{x}^r) = 
\int_{\vc{x}^r} \mu(\vc{z}) d\vc{z} = M(\vc{z}) \Big|_{\vc{z}\in\vc{x}^r}
\quad
\text{ with }
\quad
M : X \to \R^{\scriptscriptstyle^{\geq 0}}.
\end{equation}
Since $\mu$ is a probability density function, we must have $\int_A d\mu = 
1$.
For convenience, let $\hat \mu$ denote an un-normalized version of $\mu$, and 
similarly for $\hat M$.

Using the linearity of the integral, one can use linear combination
of simple functions for which exact solutions are known.
We can also construct more complex measures by partitioning 
$A$ into disjoint subsets.  In this case, 
however, we add an additional restriction in order to be sure that exact 
solutions can always be found:
We $\mu$-partition $A$ 
into subsets $S_1,\,\ldots,\,S_\sigma$, such that the boundaries of each $S_s$ 
fall on the initial set of grid lines.  Assume that 
for each set $S_s$, there exists a density function $\hat \mu_s$ that is 
exactly solvable on $S_s$.
From these, we consider $\hat \mu$ (and $\hat M$)
to be the piecewise functions defined on each $S_s$ as 
$\hat \mu_s$ (and $\hat M_s$, respectively).

Most of our computations were performed in two-dimensions. For such problems, 
given iteration $r$ and $\vc{x} = (x_1,\,x_2) \in A$, $\hat \mu(\vc{x}^r)$ 
can be written as
\begin{align}
\begin{split}
\hat \mu(\vc{x}^r) = 
&= \phantom{-}
  \hat M(x_1+w_r/2,\,x_2+w_r/2) - \hat M(x_1+w_r/2,\,x_2-w_r/2) \\
&\phantom{=}
- \hat M(x_1-w_r/2,\,x_2+w_r/2) + \hat M(x_1-w_r/2,\,x_2-w_r/2).
\end{split}
\end{align}
The closed-form choices used in our numerical tests are shown in 
Table~\ref{tb:closedMU}.
As described above, we used the table entries as building blocks in the 
construction of more complex measures.

\begin{table}[htpb]
\centering
\caption{Closed-form options for $\mu$}\label{tb:closedMU}%
\begin{tabular}{| m{3.5cm} m{0.8cm} || m{5cm} |}
\hline
$\hat \mu((x_1,\,x_2)) = 1$
& &
\vspace*{-6pt}
\begin{minipage}[t][0.2in][c]{5cm}
$\begin{aligned}
\hat M(u,\,v) = uv
\end{aligned}$
\end{minipage}%
\\\hline
$\hat \mu((x_1,\,x_2)) = x_1^tx_2^t$, & $t > 0$
&
\vspace*{-6pt}
\begin{minipage}[t][0.2in][c]{5cm}
$\begin{aligned}
\hat M(u,\,v) = (t+1)^{-2} (u^{t+1}v^{t+1})
\end{aligned}$
\end{minipage}%
\\\hline
$\hat \mu((x_1,\,x_2)) = e^{tx_1}$, & $t \neq 0$
&
\vspace*{-6pt}
\begin{minipage}[t][0.2in][c]{5cm}
$\begin{aligned}
\hat M(u,\,v) = t^{-1}ve^{tu}
\end{aligned}$
\end{minipage}%
\\\hline
$\hat \mu((x_1,\,x_2)) = e^{tx_2}$, & $t \neq 0$
&
\vspace*{-6pt}
\begin{minipage}[t][0.2in][c]{5cm}
$\begin{aligned}
\hat M(u,\,v) = t^{-1}ue^{tv}
\end{aligned}$
\end{minipage}%
\\\hline
\end{tabular}
\end{table}

\subsubsection{The closed-form Wasserstein distance over 
$\vc{x}^r$}\label{s:compWass}
We performed many tests where $\mu$ could be computed exactly but the \Was{} 
distance could not; see \sref{s:results} for details. In such cases, we made 
no attempt to approximate $P^*$, choosing instead to focus on the accuracy of 
the $\mu$-partition generated by the approximate shift set $\set{ \tilde{a}_i 
}_{i=1}^n$.

However, there were a number of cases in two dimensions where the choice of 
$\mu$ and $c$ allowed for closed-form computations.
In those cases, because the combination of $c$ and $\mu$ gives us an exact 
solution, there exists $C: X \times Y \to 
\R^{\scriptscriptstyle^{\geq 0}}$ such that
\begin{equation}
\int_{\vc{x}^r} c(\vc{z},\,\vc{y}_i) \,d\mu(\vc{z}) =
C(\vc{z},\,\vc{y}_i) \Big|_{\vc{z} \in \vc{x}^r}.
\end{equation}
As in \sref{s:comp_mass}, we write $\hat C$ when working with $\hat \mu$.

Now consider
$X,\,Y \subset \R^2$, $\vc{x} = (x_1,\,x_2) 
\in A$, and $\vc{y} = (y_1,\,y_2) \in \set{ \vc{y}_i }_{i=1}^n$.
When $\mu(\vc{x}^r) = 0$, the \Was{} distance on $\vc{x}^r$ is also 
zero.
For those boxes where $\mu(\vc{x}^r) > 0$, we can take 
advantage of the uniformity to define the function $\hat C$ in terms of a 
single 
variable: the component-wise distance between points given by 
$(\mathrm{\Delta}_1,\,\mathrm{\Delta}_2)$, where $\mathrm{\Delta}_1 = \abs*{x_1 
- y_1}$, $\mathrm{\Delta}_2 = 
\abs*{x_2 - y_2}$.
When the \Was{} distance over $\vc{x}^r$ can be computed and is non-zero, it 
takes the form
\begin{align}
\begin{split}
\int_{\vc{x}^r} c(\vc{z},\,\vc{y}) \,d\hat \mu(\vc{z})
&= \phantom{-}
  \hat C(\mathrm{\Delta}_1+w_r/2,\,\mathrm{\Delta}_2+w_r/2) - \hat 
C(\mathrm{\Delta}_1+w_r/2,\,\mathrm{\Delta}_2-w_r/2) \\
&\phantom{=}
- \hat C(\mathrm{\Delta}_1-w_r/2,\,\mathrm{\Delta}_2+w_r/2) + \hat 
C(\mathrm{\Delta}_1-w_r/2,\,\mathrm{\Delta}_2-w_r/2),
\end{split}
\end{align}
where
$\hat C : \R^2 \to \R^{\scriptscriptstyle^{\geq 0}}$ is an explicit function.

Table~\ref{tb:uniWass} gives \Was{} distance functions $\hat C$ for $c$ the 
$2$-norm
and the $p$-th power of some $p$-norm ($p\in[1,\infty)$).
By leveraging the linearity of the integral and subdividing $A$ into 
disjoint sets,
we can build combinations of ground costs and 
measures with closed form $C$.
We used this to perform tests in $\R^2$, with $\mu$ being either 
uniform or zero in relevant boxes.

\begin{table}[htpb]
\centering
\caption{Closed-form options for $C$ when $\mu$ is uniform or zero on 
$A$}\label{tb:uniWass}%
\begin{tabular}
{| m{1.7cm} || m{7.2cm} |}
\hline
$c$ & $\hat C(u,\,v)$ \rule{0pt}{2.5ex}
\\\hline\hline
$2$-norm
&
\vspace*{-10pt}
\begin{minipage}[t][1in][c]{7.2cm}
$\begin{aligned}
{\vphantom{\displaystyle\int\limits_0^1}}
\left\{
\begin{array}{ l  l }
   \frac{1}{6} u^3\log\left( \sqrt{u^2+v^2} + v \right)
\\ \quad+ \frac{1}{3} uv\sqrt{u^2 + v^2}
&\text{if} (u,\,v) \neq \vc{0}
\\ \quad+ \frac{1}{6} v^3\log\left( \sqrt{u^2+v^2} + u \right)
\\[12pt]
0 &\text{if} (u,\,v) = \vc{0}
\end{array}
\right.
\end{aligned}$
\end{minipage}%
\\\hline
\begin{minipage}[c][0.3in][c]{1in}
$p$-th power\\$p$-norm
\end{minipage}
&
\vspace*{-6pt}
\begin{minipage}[t][0.2in][c]{7.2cm}
$\begin{aligned}
(p+1)^{-1} (u^{p+1}v + uv^{p+1})
\end{aligned}$
\end{minipage}%
\\\hline
\end{tabular}
\end{table}

\subsection{Accuracy of the \Was{} distance}\label{s:Was_approx}
\subsubsection{When exact values are 
known.}\label{s:Was_known}
When $\mu$ is uniform and the region boundaries are known exactly, we can
use the formulas for the ground cost given in Table~\ref{tb:uniWass} to
compute exact values.
The $2$-norm was used for the problems shown in
Figures~\ref{f:exact}\subref{f:diag_bisect} and
\ref{f:exact}\subref{f:4by4grid}, and the $1$-norm was used to generate
\fsubref{f:exact}{f:badManhattan}.

\begin{figure}[htpb]
  \centering
  \subfloat[$2$ points: NW-SE line]{%
    \centering
    \label{f:diag_bisect}
    \resizebox{0.3\textwidth}{!}{%
    \begin{overpic}[width=\textwidth]{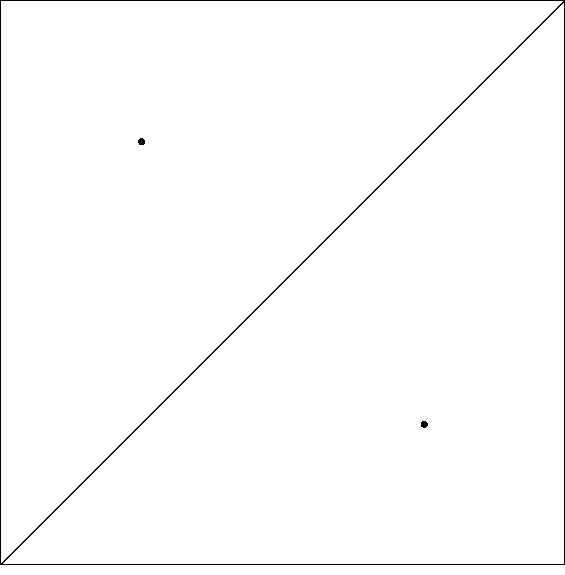}
    \put (22.5,70) {\scalebox{2.0}{$\vc{y}_0$}}
    \put (72.5,20) {\scalebox{2.0}{$\vc{y}_1$}}
    \end{overpic}}
  }%
  \subfloat[$4 \times 4$ grid arrangement]{%
    \label{f:4by4grid}
    \centering
    \resizebox{0.3\textwidth}{!}{%
    \begin{overpic}[width=\textwidth]{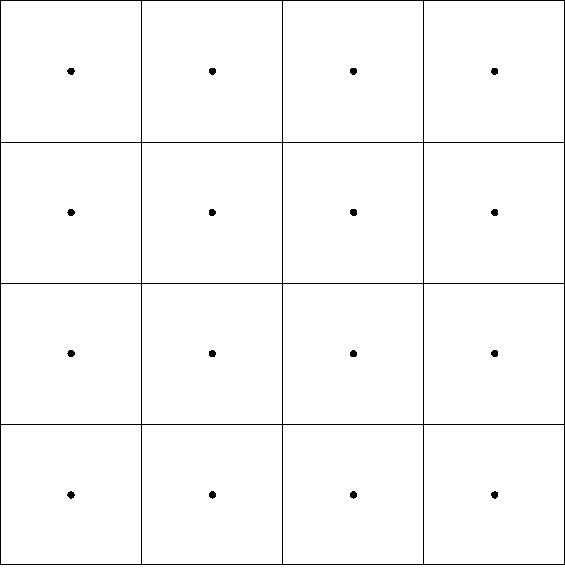}
    \put (10,08) {\scalebox{2.0}{$\vc{y}_0$}}
    \put (10,33) {\scalebox{2.0}{$\vc{y}_1$}}
    \put (10,58) {\scalebox{2.0}{$\vc{y}_2$}}
    \put (10,83) {\scalebox{2.0}{$\vc{y}_3$}}
    \put (35,08) {\scalebox{2.0}{$\vc{y}_4$}}
    \put (35,33) {\scalebox{2.0}{$\vc{y}_5$}}
    \put (35,58) {\scalebox{2.0}{$\vc{y}_6$}}
    \put (35,83) {\scalebox{2.0}{$\vc{y}_7$}}
    \put (60,08) {\scalebox{2.0}{$\vc{y}_8$}}
    \put (60,33) {\scalebox{2.0}{$\vc{y}_9$}}
    \put (60,58) {\scalebox{2.0}{$\vc{y}_{10}$}}
    \put (60,83) {\scalebox{2.0}{$\vc{y}_{11}$}}
    \put (85,08) {\scalebox{2.0}{$\vc{y}_{12}$}}
    \put (85,33) {\scalebox{2.0}{$\vc{y}_{13}$}}
    \put (85,58) {\scalebox{2.0}{$\vc{y}_{14}$}}
    \put (85,83) {\scalebox{2.0}{$\vc{y}_{15}$}}
    \end{overpic}}
  }%
  \subfloat[$2$ points: ``bad'' $1$-norm]{%
    \centering
    \label{f:badManhattan}
    \resizebox{0.3\textwidth}{!}{%
    \begin{overpic}[width=\textwidth]{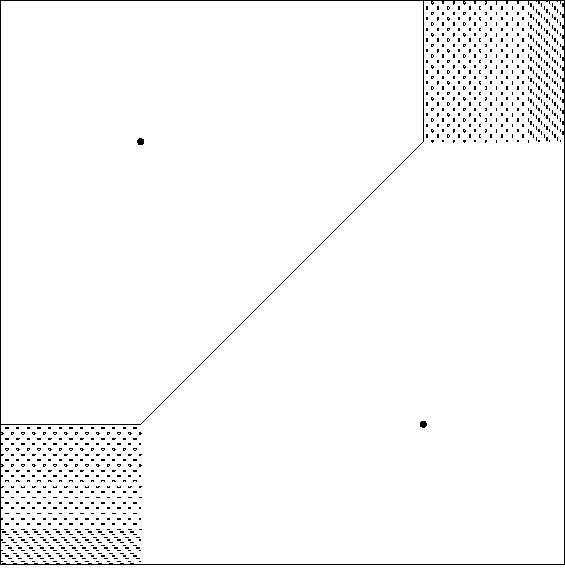}
    \put (22.5,70) {\scalebox{2.0}{$\vc{y}_0$}}
    \put (72.5,20) {\scalebox{2.0}{$\vc{y}_1$}}
    \end{overpic}}
  }%
\caption{Problems where the exact \Was{} distance and set of shifts are 
known}\label{f:exact}
\end{figure}
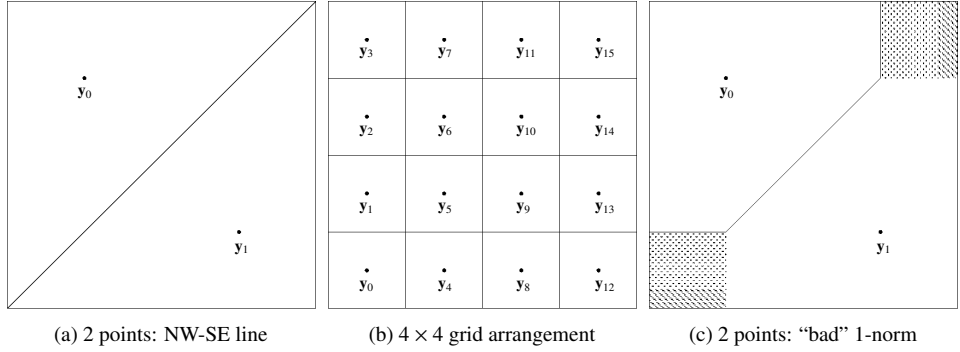

For two points on the Northwest-Southeast diagonal, placed as shown in 
\fsubref{f:exact}{f:diag_bisect}, the exact \Was{} distance is equal to
\begin{align*}
P^*_{\mathrm{NWSE}}
:= &\,\,\frac{1}{96} \left[\sqrt{2} + 7\sqrt{10} + \sinh^{-1}(1)
+ 2\sqrt{2} \sinh^{-1}(2) + \sinh^{-1}(3) \right]
\\\approx &\,\,0.3159707808963017.
\end{align*}
Table~\ref{tb:WasExactErrs}\subref{tb:NWSE} shows the absolute error
for various $w_*$.
Note that the actual decrease in error is roughly quadratic in $w_*$:
$\abs{\widetilde{P}^*_{\mathrm{NWSE}}-P^*_{\mathrm{NWSE}}} \approx 2.122 
(w_*)^{1.995}$.

\begin{table}[htpb]
\centering
\caption{\Was{} errors for the NW-SE, $4\times4$, and ``bad'' $1$-norm 
problems}\label{tb:WasExactErrs}%
\subfloat[NW-SE errors]{%
\label{tb:NWSE}
\centering
\begin{tabular}{c c}
\hline
$w_*$ & abs.\ error 
\\\hline\hline
$2^{-9^{\phantom{1}}}$  & $8.42 \times 10^{-6\phantom{1}}$
\\
$2^{-10}$  & $2.11 \times 10^{-6\phantom{1}}$ 
\\
$2^{-11}$  & $5.27 \times 10^{-7\phantom{1}}$
\\
$2^{-12}$  & $1.32 \times 10^{-7\phantom{1}}$ 
\\
$2^{-13}$  & $3.30 \times 10^{-8\phantom{1}}$ 
\\
$2^{-14}$  & $8.24 \times 10^{-9\phantom{1}}$ 
\\
$2^{-15}$  & $2.06 \times 10^{-9\phantom{1}}$ 
\\
$2^{-16}$  & $5.15 \times 10^{-10}$ 
\\
$2^{-17}$ & $1.29 \times 10^{-10}$
\\\hline
\end{tabular}
}%
\subfloat[$4\times4$ errors]{%
\label{tb:4x4}
\centering
\begin{tabular}{c c}
\hline
$w_*$ & abs.\ error 
\\\hline\hline
$2^{-9^{\phantom{1}}}$   & $2.02 \times 10^{-5\phantom{1}}$ 
\\
$2^{-10}$  & $5.04 \times 10^{-6\phantom{1}}$ 
\\
$2^{-11}$  & $1.26 \times 10^{-6\phantom{1}}$ 
\\
$2^{-12}$  & $3.15 \times 10^{-7\phantom{1}}$ 
\\
$2^{-13}$  & $7.88 \times 10^{-8\phantom{1}}$ 
\\
$2^{-14}$  & $1.97 \times 10^{-8\phantom{1}}$ 
\\
$2^{-15}$  & $4.93 \times 10^{-9\phantom{1}}$ 
\\
$2^{-16}$  & $1.23 \times 10^{-9\phantom{1}}$ 
\\
$2^{-17}$  & $3.08 \times 10^{-10}$
\\\hline
\end{tabular}
}%
\subfloat[``Bad'' $1$-norm errors]{%
\label{tb:badManhattan}
\centering
\begin{tabular}{c c}
\hline
$w_*$ & abs.\ error 
\\\hline\hline
$2^{-9^{\phantom{1}}}$  & $8.66 \times 10^{-6\phantom{1}}$
\\
$2^{-10}$  & $2.16 \times 10^{-6\phantom{1}}$ 
\\
$2^{-11}$  & $5.40 \times 10^{-7\phantom{1}}$
\\
$2^{-12}$  & $1.35 \times 10^{-7\phantom{1}}$ 
\\
$2^{-13}$  & $3.32 \times 10^{-8\phantom{1}}$ 
\\
$2^{-14}$  & $8.26 \times 10^{-9\phantom{1}}$ 
\\
$2^{-15}$  & $2.06 \times 10^{-9\phantom{1}}$ 
\\
$2^{-16}$  & $5.15 \times 10^{-10}$ 
\\
$2^{-17}$ & $1.29 \times 10^{-10}$
\\\hline
\end{tabular}
}%
\end{table}

When we have a $4\times4$ arrangement of boxes, with each $\vc{y}_i$ in the 
center, as shown in  \fsubref{f:exact}{f:4by4grid}, the exact \Was{} distance 
is 
equal to
\begin{equation*}
P^*_{4\times4}
:= \frac{1}{24}\left[ \sqrt{2} + \sinh^{-1}(1) \right]
\approx 0.09564946455802659.
\end{equation*}
Table~\ref{tb:WasExactErrs}\subref{tb:4x4} shows the error.
Again, the observed error decrease is 
roughly quadratic in $w_*$:
$\abs{\widetilde{P}^*_{4\times4}-P^*_{4\times4}} \approx 5.254 (w_*)^{1.999}$.

Recall that the theorems presented in \sref{s:math} offer no
convergence guarantee for the behavior of the $1$-norm.
In fact, the optimal solution may not be $\mu$-a.e.\ unique,
in which case the set $\set{A_i}_{i=1}^n$ may not partition $A$.
This is exactly what happens for the problem shown in
\fsubref{f:exact}{f:badManhattan}.
The problem is identical to that shown in \fsubref{f:exact}{f:diag_bisect},
except that $c$ is the $1$-norm.
Because of the change in norms, the Northeast and Southwest corners of
$A$ do not have unique transport destinations.
The loss of $\mu$-a.e.\ uniqueness, and resulting failure to partition,
is clearly visible in the figure.
However, the exact \Was{} distance can still be computed, and is equal to
\begin{align*}
P^*_{\mathrm{bad}}
:= &\,\,\frac{19}{48}.
\end{align*}
Table~\ref{tb:WasExactErrs}\subref{tb:badManhattan} shows the error.
Even though the theorems do not guarantee convergence, and
partitioning fails,
the decrease in error is nonetheless quadratic in $w_*$:
$\abs{\widetilde{P}^*_{\mathrm{bad}}-P^*_{\mathrm{bad}}} \approx 2.575 
(w_*)^{2.016}$.

For all three problems, we know the exact shift values: since every point in 
$A$ goes to the nearest $\vc{y}_i$, the shift differences are all zero, which 
means every shift should be identical.
In the $4\times4$ and ``bad'' $1$-norm problems, the shift values are
identical for every choice of $w_*$.
For the $4\times4$ problem, this is a result of computing regions that
exactly correspond to the structure of our grid.
For the ``bad'' $1$-norm problem, the exactness derives from the
relative simplicity of ground cost computations.
The shift values for the NWSE problem have an error whose decrease is roughly 
linear 
with respect to $w_*$:
$\abs{\tilde{a}_2-\tilde{a}_1} \approx 0.339 (w_*)^{1.008}$.
When $w_*=2^{-16}$, this shift error is $4.83 \times 10^{-6}$.

\subsubsection{When exact values are not 
known.}\label{s:Was_unknown}
As \sref{s:Wass_err} shows, even if the \Was{} distance is unknown, the 
\Was{} approximation error at the end of the $r$-th iteration is bounded above 
by
\begin{equation}\label{e:WasErrComputed}
\sum_{\vc{x} \in B^r} \mu(\vc{x}^r)
\max_{\vc{x}_0 \in \vc{x}^r} g_{ij}(\vc{x}_0),
\end{equation}
where $i$ and $j$, $i \neq j$ refer to the destinations of $\vc{x}$ and some 
neighbor.
In practice, we can use continuity to refine that estimate still further,
as described in \sref{s:wass}.

As $w_r \to 0$, $\max_{\vc{x}_0 \in \vc{x}^r} g_{ij}(\vc{x}_0) \to 
\abs{a_{ij}}$ for each $i \neq j$, and $\mu(\cl{B}^r) \to 0$.
If the boundary method is working effectively, we can expect to 
see the \Was{} distance error decreasing with respect to $\mu(\cl{B}^r)$. If 
$\mu$ is uniform on $A$, that decrease should be linear with respect to the 
volume $\abs{\cl{B}^r}$.

We considered the change in the computed \Was{} distance for Example 
\ref{x:steps}
using three canonical ground costs: the $1$-norm, the $2$-norm, and 
the squared $2$-norm. The resulting $\mu$-partitions are shown in 
Figures~\ref{f:norm}\subref{f:norm1}, \ref{f:meas_comp}\subref{f:uni}, and 
\ref{f:odd_norms}\subref{f:norm2sq}, respectively.
Since $\mu$ is uniform, and $A=[0,\,1]^2$, \tref{t:WassErr} suggests that we
should see a quadratic convergence for the $2$-norm.
(The theorem makes no convergence claim for the $1$-norm or the squared 
$2$-norm.)

For each ground cost, we computed the \Was{} approximation error in two ways:
\begin{enumerate}
\item The worst-case \Was{} distance error bound, given by applying 
\eqref{e:WasErrComputed}.
\item The rate of change with respect to a reference approximation,
\begin{equation*}
\mathrm{\Delta} \widetilde{P}^*_{16}(w_*) :=
\mathrm{\Delta} \widetilde{P}^*_{16}(2^{-m}) =
\frac{\abs*{\widetilde{P}^*_m - \widetilde{P}^*_{16}}}
{\abs*{\widetilde{P}^*_{m+1} - \widetilde{P}^*_{16}}},
\quad
\text{ for }
m < 15.
\end{equation*}
\end{enumerate}
For all three ground cost functions, the worst-case \Was{} distance
error is roughly linear in $w_*$, and the 
rate of change of $\mathrm{\Delta} \widetilde{P}^*_{16}$ is roughly
quadratic in $w_*$.
Though \tref{t:WassErr} only guarantees quadratic convergence for
the $2$-norm, we also observe quadratic convergence for
the $1$-norm and squared $2$-norm.
For this example, the $1$-norm generated a $\mu$-a.e.\ unique partition, but
comparable convergence was seen in
tests where partitioning failed.
Results are given in Table~\ref{tb:WasApproxDataEQ}.

\begin{table}[htpb]
\centering
\caption{\Was{} approximation behavior with respect 
to $w_*$}\label{tb:WasApproxDataEQ}
\renewcommand{\arraystretch}{1.5}
\begin{tabular}{| c || c | c | c |}
\hline
$c$ & $\widetilde{P}^*_{16}$ & $\mathrm{err}_{\max}(w_*)$
& $\mathrm{\Delta} \widetilde{P}^*_{16}(w_*)$
\\\hline\hline
$1$-norm & $0.25702262181$
& $0.457(w_*)^{1.020}$
& $1.186(w_*)^{2.025}$
\\\hline
$2$-norm & $0.20754605961 $
& $0.361(w_*)^{1.008}$
& $4.151(w_*)^{2.023}$
\\\hline
squared $2$-norm & $0.05290682486 $
& $0.221(w_*)^{1.008}$
& $2.668(w_*)^{2.029}$
\\\hline
\end{tabular}
\end{table}

\subsection{$\mu$-Partitions in $\R^2$}\label{s:R2partitions}
\subsubsection{Uniform and non-uniform measures $\mu$ and 
$\nu$}\label{s:meas_options}
We include three examples with variations of $\mu$ and $\nu$, shown in 
\fref{f:meas_comp}.
All three assume $c$ is the $2$-norm.

\begin{figure}[htpb]
  \centering
  \subfloat[$\mu$ and $\nu$ uniform]{%
    \label{f:uni}
    \centering
    \resizebox{0.3\textwidth}{!}{%
    \begin{overpic}[width=\textwidth]{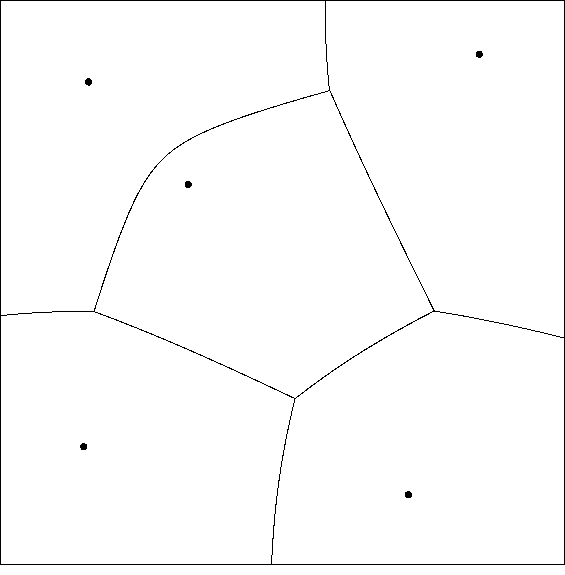}
    \put (11,81) {\scalebox{2.0}{$\vc{y}_0$}}
    \put (81,86) {\scalebox{2.0}{$\vc{y}_1$}}
    \put (29,63) {\scalebox{2.0}{$\vc{y}_2$}}
    \put (10,17) {\scalebox{2.0}{$\vc{y}_3$}}
    \put (68,08) {\scalebox{2.0}{$\vc{y}_4$}}
    \end{overpic}}
  }%
  \subfloat[$\nu$ non-uniform]{%
    \label{f:nu_non}
    \centering
    \resizebox{0.3\textwidth}{!}{%
    \begin{overpic}[width=\textwidth]{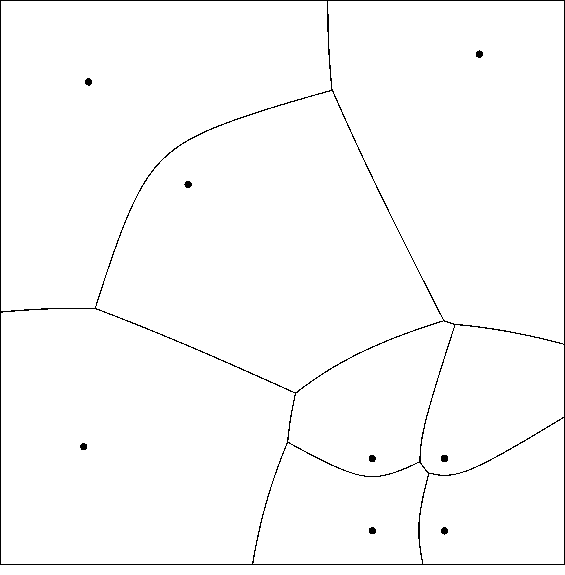}
    \put (11,81) {\scalebox{2.0}{$\vc{y}_0$}}
    \put (81,86) {\scalebox{2.0}{$\vc{y}_1$}}
    \put (29,63) {\scalebox{2.0}{$\vc{y}_2$}}
    \put (10,17) {\scalebox{2.0}{$\vc{y}_3$}}
    \put (60,03) {\scalebox{2.0}{$\vc{y}_4$}}
    \put (80,21) {\scalebox{2.0}{$\vc{y}_5$}}
    \put (80,03) {\scalebox{2.0}{$\vc{y}_6$}}
    \put (60,24) {\scalebox{2.0}{$\vc{y}_7$}}
    \end{overpic}}
  }%
  \subfloat[$\mu(x_1,\,x_2) = x_1x_2$]{%
    \label{f:mu_xy}
    \centering
    \resizebox{0.3\textwidth}{!}{%
    \begin{overpic}[width=\textwidth]{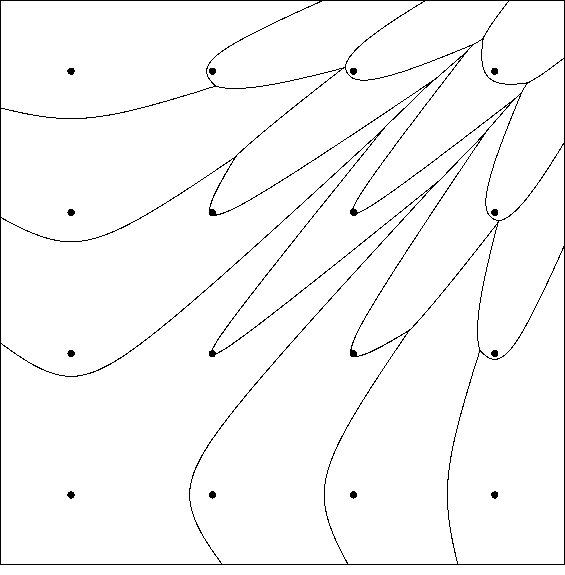}
    \put (08,08) {\scalebox{2.0}{$\vc{y}_0$}}
    \put (10,40) {\scalebox{2.0}{$\vc{y}_1$}}
    \put (10,65) {\scalebox{2.0}{$\vc{y}_2$}}
    \put (10,83) {\scalebox{2.0}{$\vc{y}_3$}}
    \put (40,11) {\scalebox{2.0}{$\vc{y}_4$}}
    \put (33,33) {\scalebox{2.0}{$\vc{y}_5$}}
    \put (33,58) {\scalebox{2.0}{$\vc{y}_6$}}
    \put (40,87) {\scalebox{2.0}{$\vc{y}_7$}}
    \put (65,11) {\scalebox{2.0}{$\vc{y}_8$}}
    \put (58,33) {\scalebox{2.0}{$\vc{y}_9$}}
    \put (56,58) {\scalebox{2.0}{$\vc{y}_{10}$}}
    \put (65,89) {\scalebox{2.0}{$\vc{y}_{11}$}}
    \put (90,11) {\scalebox{2.0}{$\vc{y}_{12}$}}
    \put (90,33) {\scalebox{2.0}{$\vc{y}_{13}$}}
    \put (78,60) {\scalebox{2.0}{$\vc{y}_{14}$}}
    \put (89,90) {\scalebox{2.0}{$\vc{y}_{15}$}}
    \end{overpic}}
  }%
\caption{Partitions for uniform and non-uniform 
measures}\label{f:meas_comp}
\end{figure}
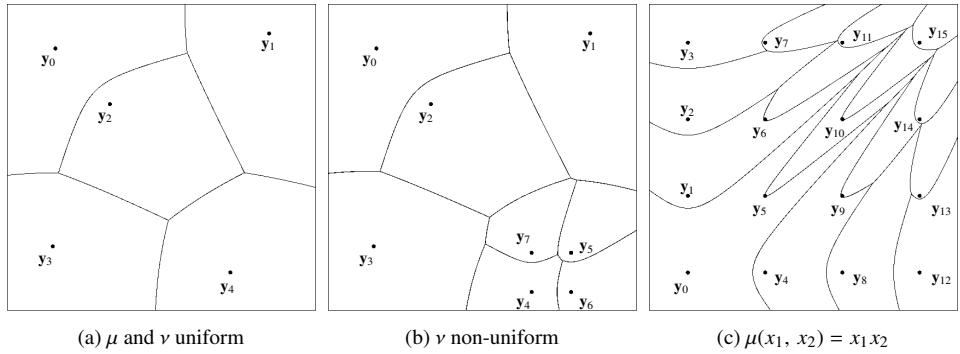

In \fsubref{f:meas_comp}{f:uni}, we assume $\mu$ is the uniform 
continuous probability 
distribution on $A$, and $\nu$ is the uniform discrete distribution with 
$n=5$.
The five points where $\nu = 1/5$ are placed in the positions used 
in~\cite{Barrett2007a}.
\fsubref{f:meas_comp}{f:uni} shows the $\mu$-partition obtained by the boundary 
method; for 
comparison, see  \cite[Figure 3 (right)]{Barrett2007a}.

Starting from the points shown in \fsubref{f:meas_comp}{f:uni}, we next take 
the 
point 
$\vc{y}_4$ and split it into four new points, each of one quarter-mass, 
positioned equidistantly from the point's original location.
This gives us a non-uniform $\nu$ with four points of weight $1/5$ and four 
of weight $1/20$.
We keep $\mu$ uniform.
The resulting $\mu$-partition is shown in \fsubref{f:meas_comp}{f:nu_non}.

In \fsubref{f:meas_comp}{f:mu_xy}, we choose the nonuniform probability density
$\mu(x_1,\,x_2) = \frac{1}{4}x_1x_2$.
For $\nu$, we choose the uniform $4 \times 4$ grid of points given in 
\fsubref{f:exact}{f:4by4grid}. By comparing the results in Figures 
\ref{f:exact}\subref{f:4by4grid} and 
\ref{f:meas_comp}\subref{f:mu_xy}, the impact of $\mu$'s nonuniformity becomes 
obvious.
While the individual regions no longer have equal Lebesgue measure, each has 
equal $\mu$-measure $1/16$.
The larger regions in the lower-left correspond to the lower density of $\mu$ 
in that corner, while the smaller regions in the upper-right correspond to the 
higher concentration of $\mu$-density there.

\subsubsection{Discontinuous and zero-measure 
$\mu$}\label{s:bad_mu}
Next, we deliberately introduce a discontinuous $\mu$ that is not strictly 
positive:
\begin{equation}
\mu(\vc{x}) =
\begin{cases}
0 & \text{ if } \vc{x} \in [0,\,1/2]^2 \\
4/3 & \text{ otherwise. }
\end{cases}
\end{equation}
We still have
$\int_A \,d\mu(\vc{x}) = 1$,
so $\mu$ is a probability density function on $A = [0,1]^2$.
For $\nu$, we use the uniform $4\times4$ grid shown in 
\fsubref{f:exact}{f:4by4grid}.
\fref{f:zero} shows the results.

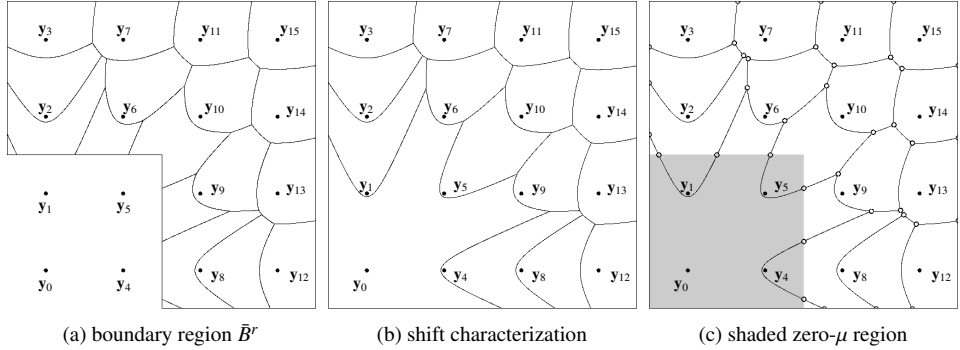
\begin{figure}[htpb]
  \centering
  \subfloat[boundary region $\cl{B}^r$]{%
    \label{f:zero_bnd}
    \centering
    \resizebox{0.3\textwidth}{!}{%
    \begin{overpic}[width=\textwidth]{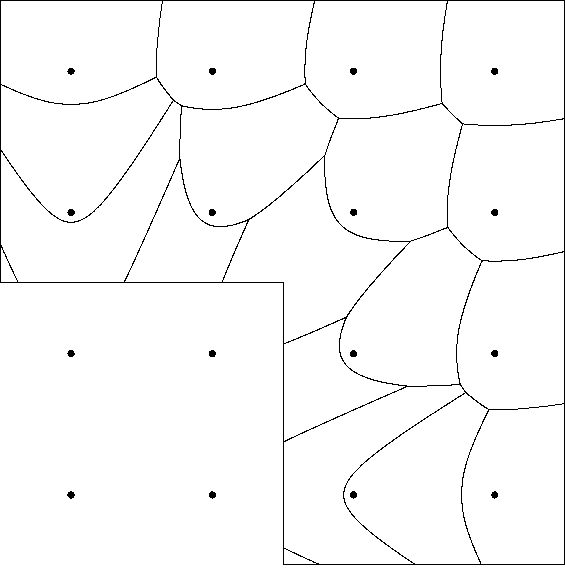}
    \put (10,08) {\scalebox{2.0}{$\vc{y}_0$}}
    \put (10,33) {\scalebox{2.0}{$\vc{y}_1$}}
    \put (10,65) {\scalebox{2.0}{$\vc{y}_2$}}
    \put (10,90) {\scalebox{2.0}{$\vc{y}_3$}}
    \put (35,08) {\scalebox{2.0}{$\vc{y}_4$}}
    \put (35,33) {\scalebox{2.0}{$\vc{y}_5$}}
    \put (37,65) {\scalebox{2.0}{$\vc{y}_6$}}
    \put (35,90) {\scalebox{2.0}{$\vc{y}_7$}}
    \put (65,11) {\scalebox{2.0}{$\vc{y}_8$}}
    \put (65,39) {\scalebox{2.0}{$\vc{y}_9$}}
    \put (63,65) {\scalebox{2.0}{$\vc{y}_{10}$}}
    \put (62,90) {\scalebox{2.0}{$\vc{y}_{11}$}}
    \put (90,11) {\scalebox{2.0}{$\vc{y}_{12}$}}
    \put (89,39) {\scalebox{2.0}{$\vc{y}_{13}$}}
    \put (89,64) {\scalebox{2.0}{$\vc{y}_{14}$}}
    \put (87,90) {\scalebox{2.0}{$\vc{y}_{15}$}}
    \end{overpic}}
  }%
  \subfloat[shift characterization]{%
    \label{f:zero_recon}
    \centering
    \resizebox{0.3\textwidth}{!}{%
    \begin{overpic}[width=\textwidth]{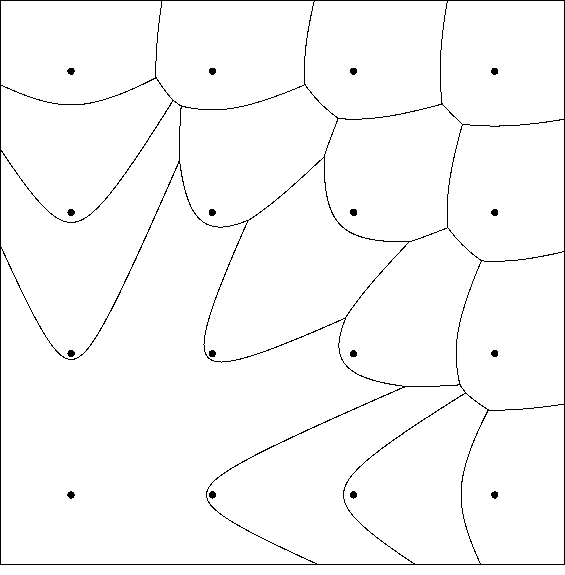}
    \put (08,08) {\scalebox{2.0}{$\vc{y}_0$}}
    \put (10,40) {\scalebox{2.0}{$\vc{y}_1$}}
    \put (10,65) {\scalebox{2.0}{$\vc{y}_2$}}
    \put (10,90) {\scalebox{2.0}{$\vc{y}_3$}}
    \put (40,11) {\scalebox{2.0}{$\vc{y}_4$}}
    \put (40,40) {\scalebox{2.0}{$\vc{y}_5$}}
    \put (37,65) {\scalebox{2.0}{$\vc{y}_6$}}
    \put (35,90) {\scalebox{2.0}{$\vc{y}_7$}}
    \put (65,11) {\scalebox{2.0}{$\vc{y}_8$}}
    \put (65,39) {\scalebox{2.0}{$\vc{y}_9$}}
    \put (63,65) {\scalebox{2.0}{$\vc{y}_{10}$}}
    \put (62,90) {\scalebox{2.0}{$\vc{y}_{11}$}}
    \put (90,11) {\scalebox{2.0}{$\vc{y}_{12}$}}
    \put (89,39) {\scalebox{2.0}{$\vc{y}_{13}$}}
    \put (89,64) {\scalebox{2.0}{$\vc{y}_{14}$}}
    \put (87,90) {\scalebox{2.0}{$\vc{y}_{15}$}}
    \end{overpic}}
  }%
  \subfloat[shaded zero-$\mu$ region]{%
    \label{f:zero_end}
    \centering
    \resizebox{0.3\textwidth}{!}{%
    \begin{overpic}[width=\textwidth]{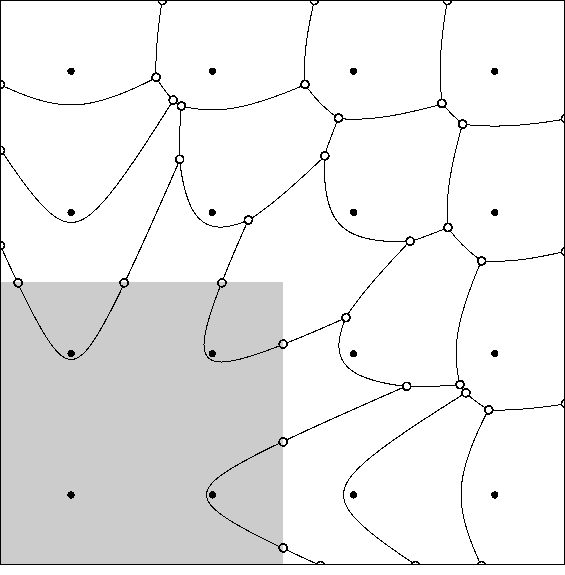}
    \put (08,08) {\scalebox{2.0}{$\vc{y}_0$}}
    \put (10,40) {\scalebox{2.0}{$\vc{y}_1$}}
    \put (10,65) {\scalebox{2.0}{$\vc{y}_2$}}
    \put (10,90) {\scalebox{2.0}{$\vc{y}_3$}}
    \put (40,11) {\scalebox{2.0}{$\vc{y}_4$}}
    \put (40,40) {\scalebox{2.0}{$\vc{y}_5$}}
    \put (37,65) {\scalebox{2.0}{$\vc{y}_6$}}
    \put (35,90) {\scalebox{2.0}{$\vc{y}_7$}}
    \put (65,11) {\scalebox{2.0}{$\vc{y}_8$}}
    \put (65,39) {\scalebox{2.0}{$\vc{y}_9$}}
    \put (63,65) {\scalebox{2.0}{$\vc{y}_{10}$}}
    \put (62,90) {\scalebox{2.0}{$\vc{y}_{11}$}}
    \put (90,11) {\scalebox{2.0}{$\vc{y}_{12}$}}
    \put (89,39) {\scalebox{2.0}{$\vc{y}_{13}$}}
    \put (89,64) {\scalebox{2.0}{$\vc{y}_{14}$}}
    \put (87,90) {\scalebox{2.0}{$\vc{y}_{15}$}}
    \end{overpic}}
  }%
\caption{$\mu$ is zero in the lower-left quadrant}\label{f:zero}
\end{figure}

In \fsubref{f:zero}{f:zero_bnd}, we see the boundary set used to generate the 
solution.
Points between regions are retained, as are points adjacent to regions of 
measure zero. No computations are done on the lower-left region, because any 
destination is equally valid on boxes of $\mu$-measure zero.

However, when the shift definition is applied to the semi-discrete optimal 
transport problem, there is only one valid shift-characterized solution over 
$A$.
\fsubref{f:zero}{f:zero_recon} shows that solution.
The unique shift differences force the selection of a unique boundary set 
$B$, even in the region where $\mu$ is zero.

\fsubref{f:zero}{f:zero_end} shows the shift characterization again, but here 
the region of $\mu$-zero measure is shaded, helping to confirm visually that 
the regions have equal $\mu$-measure.
\fsubref{f:zero}{f:zero_end} also shows the locations of intersection points we 
identified using the boundary method.
These intersections were used to accurately compute the set of shifts.

\subsubsection{Norms as ground cost 
functions}\label{s:cost_options}
The computations in Sections \ref{s:meas_options} and \ref{s:bad_mu}
all assume the ground cost function equals the $2$-norm, but as
\sref{s:Was_unknown} suggests, computation with other functions is quite
possible.
Using the same problem solved with the $2$-norm in 
\fsubref{f:meas_comp}{f:uni}, we generated
$\mu$-partitions for a wide range of $p$-norm ground costs.
Results for the $1$-norm, $10$-norm, and $\infty$-norm
are shown in \fref{f:norm}.
Note that the $1$-norm and $\infty$-norm converge to
($\mu$-a.e.\ unique) solutions, even though
those norms are not covered by our theorems.

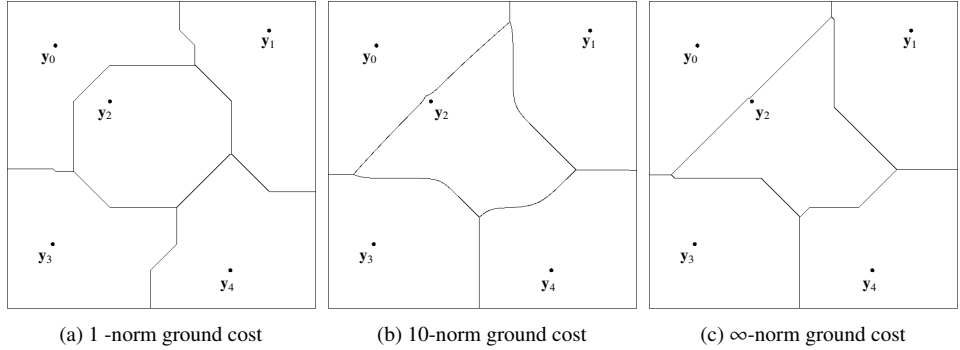
\begin{figure}[htpb]
  \centering
  \subfloat[${1}$ -norm ground cost]{%
    \label{f:norm1}
    \centering
    \resizebox{0.3\textwidth}{!}{%
    \begin{overpic}[width=\textwidth]{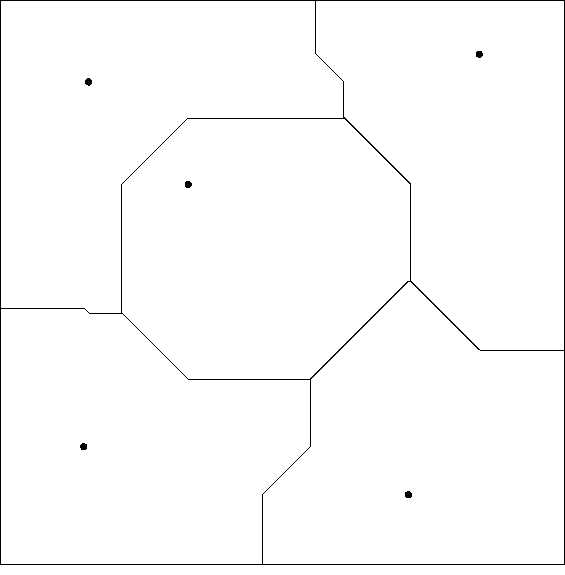}
    \put (11,81) {\scalebox{2.0}{$\vc{y}_0$}}
    \put (81,86) {\scalebox{2.0}{$\vc{y}_1$}}
    \put (29,63) {\scalebox{2.0}{$\vc{y}_2$}}
    \put (10,17) {\scalebox{2.0}{$\vc{y}_3$}}
    \put (68,08) {\scalebox{2.0}{$\vc{y}_4$}}
    \end{overpic}}
  }%
  \subfloat[${10}$-norm ground cost]{%
    \label{f:norm10}
    \centering
    \resizebox{0.3\textwidth}{!}{%
    \begin{overpic}[width=\textwidth]{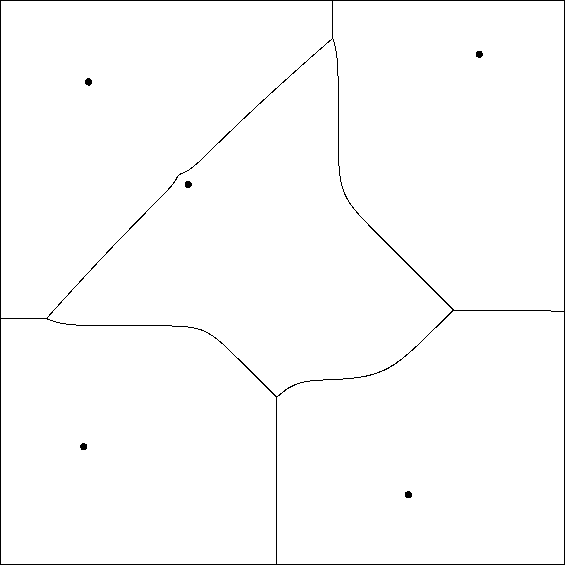}
    \put (11,81) {\scalebox{2.0}{$\vc{y}_0$}}
    \put (81,86) {\scalebox{2.0}{$\vc{y}_1$}}
    \put (34,63) {\scalebox{2.0}{$\vc{y}_2$}}
    \put (10,17) {\scalebox{2.0}{$\vc{y}_3$}}
    \put (68,08) {\scalebox{2.0}{$\vc{y}_4$}}
    \end{overpic}}
  }%
  \subfloat[${\infty}$-norm ground cost]{%
    \label{f:sup-norm}
    \centering
    \resizebox{0.3\textwidth}{!}{%
    \begin{overpic}[width=\textwidth]{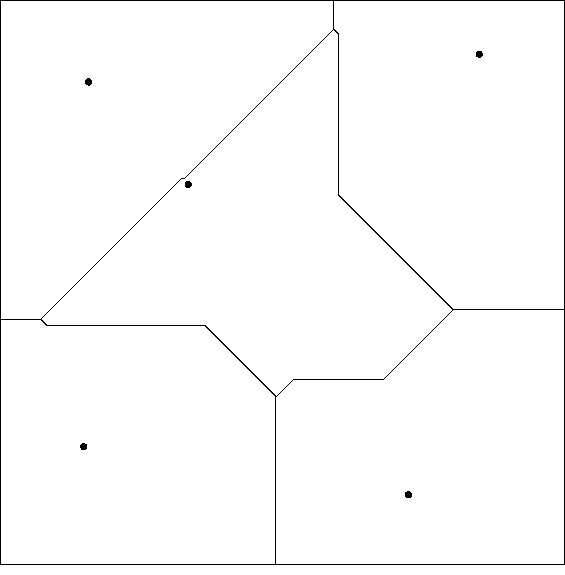}
    \put (11,81) {\scalebox{2.0}{$\vc{y}_0$}}
    \put (81,86) {\scalebox{2.0}{$\vc{y}_1$}}
    \put (34,63) {\scalebox{2.0}{$\vc{y}_2$}}
    \put (10,17) {\scalebox{2.0}{$\vc{y}_3$}}
    \put (68,08) {\scalebox{2.0}{$\vc{y}_4$}}
    \end{overpic}}
  }%
\caption{Equal area using different ground cost norms}\label{f:norm}
\end{figure}

\subsubsection{Other ground costs $c$}\label{s:weird_costs}
The computations above all assume that the ground cost function is a norm.
However, the boundary method works equally well on much more general ground 
cost 
functions.
Three examples are shown in \fref{f:odd_norms}.

\begin{figure}[htpb]
  \centering
  \subfloat[squared $2$-norm]{%
    \centering
    \label{f:norm2sq}
    \resizebox{0.3\textwidth}{!}{%
    \begin{overpic}[width=\textwidth]{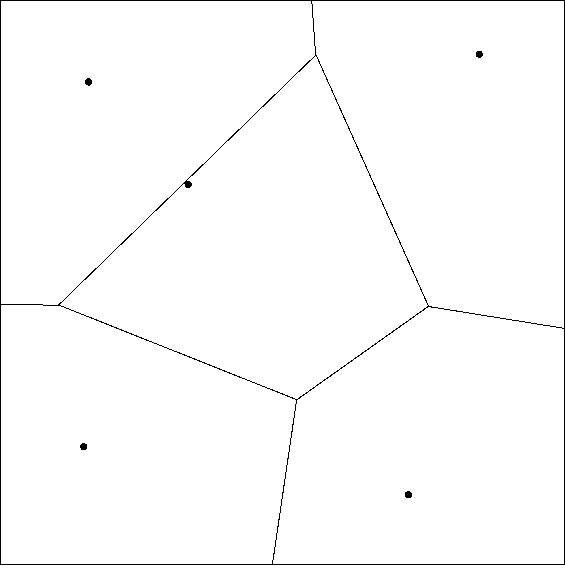}
    \put (11,81) {\scalebox{2.0}{$\vc{y}_0$}}
    \put (81,86) {\scalebox{2.0}{$\vc{y}_1$}}
    \put (34,63) {\scalebox{2.0}{$\vc{y}_2$}}
    \put (10,17) {\scalebox{2.0}{$\vc{y}_3$}}
    \put (68,08) {\scalebox{2.0}{$\vc{y}_4$}}
    \end{overpic}}
  }%
  \subfloat[$p$-function ground cost]{%
    \label{f:norm0.5}
    \centering
    \resizebox{0.3\textwidth}{!}{%
    \begin{overpic}[width=\textwidth]{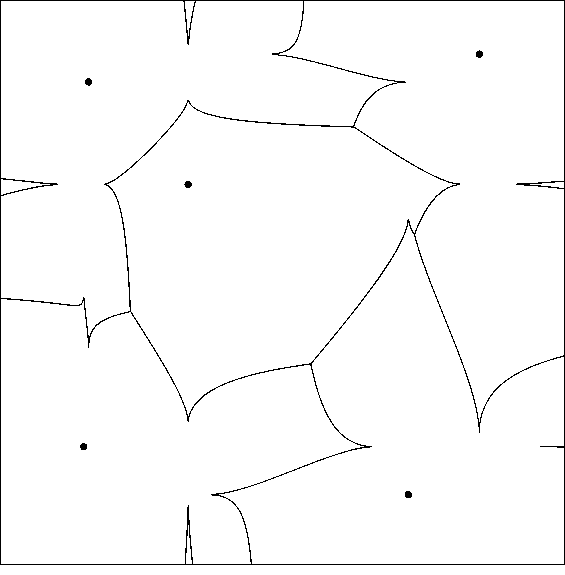}
    \put (11,81) {\scalebox{2.0}{$\vc{y}_0$}}
    \put (81,86) {\scalebox{2.0}{$\vc{y}_1$}}
    \put (29,63) {\scalebox{2.0}{$\vc{y}_2$}}
    \put (10,17) {\scalebox{2.0}{$\vc{y}_3$}}
    \put (68,08) {\scalebox{2.0}{$\vc{y}_4$}}
    \end{overpic}}
  }%
  \subfloat[$p$-polynomial ground cost]{%
    \label{f:mix-norm02}
    \centering
    \resizebox{0.3\textwidth}{!}{%
    \begin{overpic}[width=\textwidth]{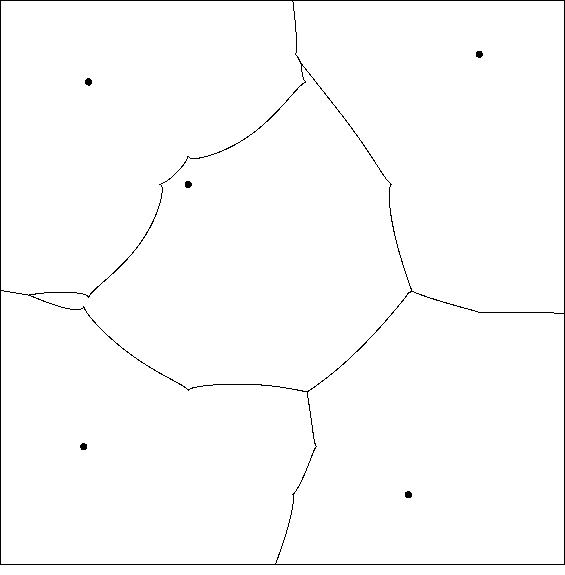}
    \put (11,81) {\scalebox{2.0}{$\vc{y}_0$}}
    \put (81,86) {\scalebox{2.0}{$\vc{y}_1$}}
    \put (34,63) {\scalebox{2.0}{$\vc{y}_2$}}
    \put (10,17) {\scalebox{2.0}{$\vc{y}_3$}}
    \put (68,08) {\scalebox{2.0}{$\vc{y}_4$}}
    \end{overpic}}
  }%
\caption{Equal area using non-norm ground 
costs}\label{f:odd_norms}
\end{figure}
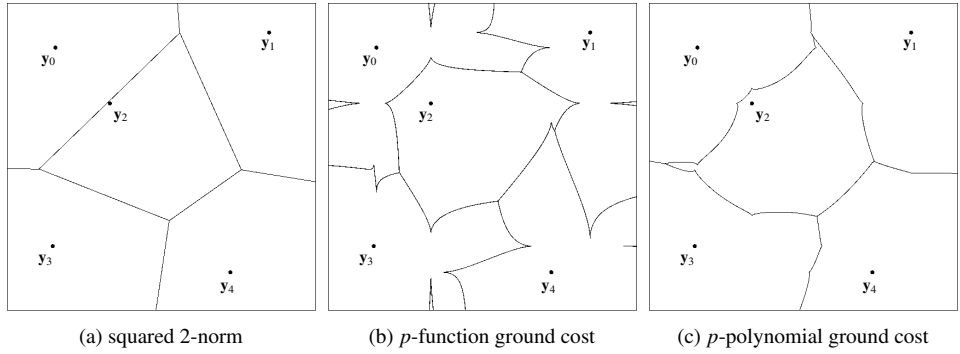

\fsubref{f:odd_norms}{f:norm2sq} shows the result given by the squared 
$2$-norm. 
Because the squared $2$-norm is not itself a norm, we were only able to make the 
most general mathematical claims regarding its behavior.
However, the boundary method has no trouble with it.
In fact, as \sref{s:Was_unknown} indicates,
its convergence behavior is practically identical to that of the $p$-norms.

When $0 < p < 1$, the $p$-norm formula can still be applied,
even though the resulting function does not satisfy the triangle inequality.
Formally, one has
\begin{equation}
c_p(\vc{x}_1,\,\vc{x}_2) :=
\begin{cases}
\sum_{k=1}^d \left[ \abs{x^2_k-x^1_k}^p \right]^{1/p} & p \in (0,\,\infty) \\
\max_{k \in \N_d} \abs{x^2_k-x^1_k} & p = \infty.
\end{cases}
\end{equation}
As \fsubref{f:odd_norms}{f:norm0.5} illustrates, even these ``$p$-function'' 
transport problems can be 
approximated. However, when $p<1$, the regions become discontinuous
and disconnected, as typified by the ``spikes'' on the exterior walls.
(The spike on the lower right is part of the region coupled with $\vc{y}_3$, 
while the other four spikes are coupled with $\vc{y}_2$.)
Note that the ${1/2}$-function is concave. Such functions are directly 
applicable to transport problems involving economies of scale; e.g.,
see~\cite{Gangbo1996a}.

\fsubref{f:odd_norms}{f:mix-norm02} shows a ground cost function defined as a 
polynomial 
combination of $p$-functions with positive coefficients:
\begin{equation*}
c(\vc{x}_1,\,\vc{x}_2) = 
4c_2(\vc{x}_1,\,\vc{x}_2)^{28/5} 
+ 61c_{1/2}(\vc{x}_1,\,\vc{x}_2).
\end{equation*}
This function, like many other ``$p$-polynomial'' functions,
is neither
convex nor concave, changing behavior with distance.

\subsection{$\mu$-Partitions in $\R^3$}\label{s:R3partitions}
As we showed in \sref{s:math}, there is no theoretical obstacle to applying the 
boundary method to higher-dimensional problems, though visual 
representation becomes more complex.

\begin{figure}[htpb]
  \centering
    \resizebox{0.50\textwidth}{!}{%
    \begin{overpic}[width=\textwidth]{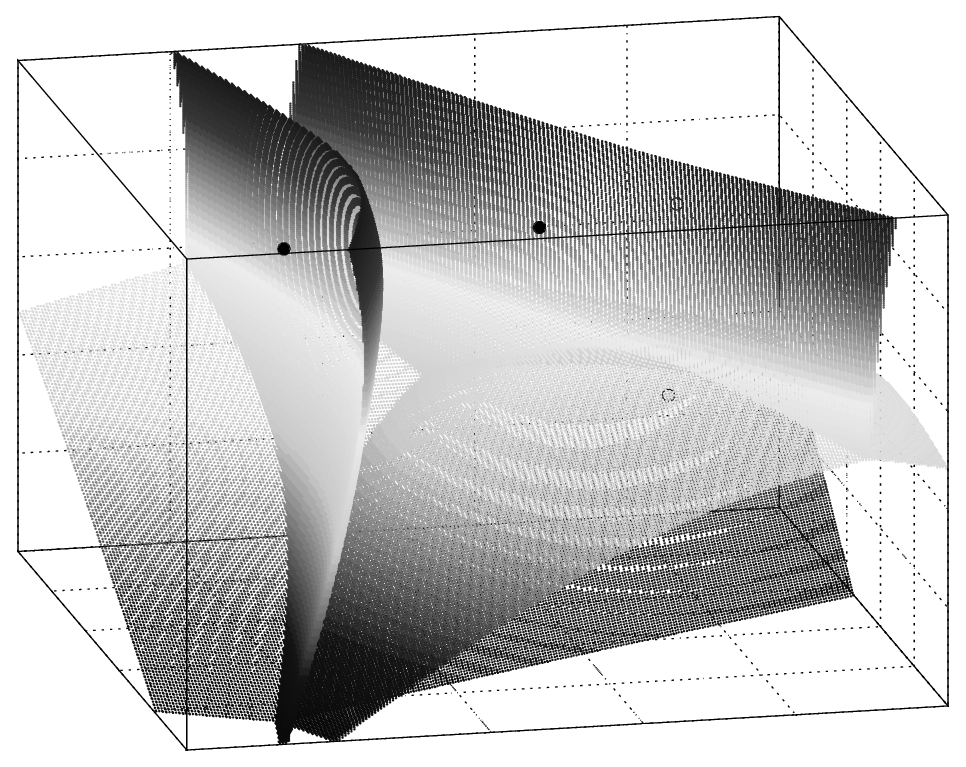}
    \end{overpic}}
\caption{Three-dimensional semi-discrete solution with $n=5$}\label{f:3d}
\end{figure}

The image in \fref{f:3d} was generated by taking $c$ to be the $2$-norm, $\mu$ 
the uniform continuous probability density, 
and $\nu$ the uniform discrete probability density with five 
randomly-placed non-zero points in $[0,\,1]^3$.
Even in this relatively simple case, it is impossible to find a single 
point-of-view that clearly shows all five non-zero points while clearly 
illustrating the boundaries of the $\mu$-partitions.
However, even though clear illustration is problematic, the computations made 
with the boundary method were completely successful.

\subsection{Scaling behavior}\label{s:scaling}
One important advantage to the boundary method is its reduction of the 
complexity of the discretized problem, compared to traditional methods.
Before considering the numerical results, it is worth developing a generalized 
comparison that puts this reduction in perspective:

\noindent
Suppose for the sake of argument that a discretization with width $2^{-M}$ is 
required to solve a problem in $\R^2$ with $N$ positive points in $Y$.
Generating the full grid would create a product space $X \times Y$ of size 
$2^{2M} N$.
Say the boundary method is used instead, with a fixed initial discretization 
width of $2^{-4}$.
Each application of \aref{a:discard} of the boundary method algorithm
removes approximately half the points in $A^r$, so by discarding interiors the 
method constructs a product space of size $2^{M+4} N$.

Assume that we compute 
solutions for both the boundary method and the full product space using the 
same 
linear solver (e.g., the network simplex method).
Using it to solve the largest boundary problem of size 
$2^{M+4} N$, we have $V = 2^{M+4} + N$ vertices and $E = 2^{M+4} N$ 
edges.
Solving over the full product space gives 
$V=2^{2M} + N$ vertices and $E = 2^{2M} N$ edges.
Hence, even if we assumed a solver with complexity $\mathcal{O}(V)$ (and no 
such 
solver exists), the ratio would be approximately $2^M$ to $M$. Typically, it 
is closer to $2^{2M}$ to $M$.

Of course, this improved complexity would be irrelevant if the constant factor 
was excessively large.
Fortunately, that this is not the case, as our next section illustrates.

Since we focus here on the semi-discrete problem, for purposes of evaluating
complexity, we assume $X$ is discretized into $W^d$ elements, and $Y$ into $N$ 
elements.
Given $W$ sufficiently large, the resulting network has
$W^d+N\sim \mathcal{O}(W^d)$ nodes and $W^dN$ arcs.

\subsubsection{Scaling on the plane with respect to $W = 1 /
w_*$}\label{s:Wscaling}
Here we consider scaling on the 
plane with respect to $ W = \frac{1}{w_*}$.
We used Example \ref{x:steps}, with $\mu$ and $\nu$ uniform and $c$ the 
$2$-norm. 
The locations of the
$5$ points where $\nu = 1/5$ were fixed as depicted in 
\fsubref{f:meas_comp}{f:uni}.
We defined target widths $w_* = 2^{-m}$, $m \in \N$, and computed the
time taken by the boundary method.
By repeating this process for a few different location sets (and averaging 
them), we estimated the average scaling behavior of the boundary method with 
respect to $W$.
The test results are shown in Table~\ref{tb:IndScaling}\subref{tb:Wscaling}, and
the scaling equations are on the left side of Table~\ref{tb:IndScaleEQ}.

\begin{table}[htpb]
\centering
\caption{Planar scaling with respect to $W$ and $N$}\label{tb:IndScaling}%
\subfloat[Scaling with respect to $W$]{%
\label{tb:Wscaling}
\centering
\begin{tabular}[b]{| c | r c r |}
\hline
& \multicolumn{3}{c|}{$N=5$}  
\\\hline
$W$ & \multicolumn{2}{c}{T (sec)} & \multicolumn{1}{c|}{S (MB)}
\\\hline\hline
$2^{12}$ & $0.855$ & & $24.540$
\\ 
$2^{13}$ & $2.005$ & & $49.100$
\\
$2^{14}$ & $4.497$ & & $98.210$
\\
$2^{15}$ & $11.025$ & & $196.400$
\\
$2^{16}$ & $28.093$ & & $394.400$
\\
$2^{17}$ & $60.577$ & & $785.800$
\\
$2^{18}$ & $132.397$ & & $1571.840$
\\
$2^{19}$ & $292.158$ & & $3151.872$
\\
$2^{20}$ & $640.660$ & & $6309.888$
\\\hline
\end{tabular}
}%
\subfloat[Scaling with respect to $N$]{%
\label{tb:Nscaling}
\centering
\begin{tabular}[b]{| c | c c | c c |}
\hline
& \multicolumn{2}{c |}{$W=2^{10}$} & \multicolumn{2}{c |}{$W=2^{11}$}
\\\hline
$N$ & T (sec) & S (MB)
& T (sec) & S (MB)
\\\hline\hline
128 & \phantom{1}6.938 & 17.25 & 22.365 & 33.91 \\
136 & 12.190 & 18.24 & 36.601 & 35.05 \\
144 & 10.982 & 17.99 & 29.952 & 36.49 \\
152 & 13.139 & 18.54 & 36.703 & 41.27 \\
160 & 11.420 & 18.66 & 34.801 & 40.27 \\
168 & 15.727 & 20.97 & 44.959 & 40.66 \\
176 & 15.332 & 21.38 & 44.873 & 43.06 \\
184 & 18.243 & 21.38 & 53.689 & 43.20 \\
192 & 12.796 & 21.60 & 40.029 & 43.66
\\\hline
\end{tabular}
}%
\end{table}

\begin{table}[htpb]
\centering
\caption{Time and storage scaling with respect to $W$ and 
$N$ separately}\label{tb:IndScaleEQ}%
\renewcommand{\arraystretch}{1.5}
\begin{tabular}{| l || c || l |}
\hline
$T(W) \approx 4.356\times10^{-5}W\ln W$
& Time
& $T(N) \approx 4.582\times10^{-2}N\ln N$
\\\hline
$S(W) \approx 6.015\times10^{-3}W$
& Storage
& $S(N) \approx 3.162\, N^{1/2}$
\\\hline
\end{tabular}
\end{table}

\subsubsection{Scaling on the plane with respect to $N$}\label{s:Nscaling}
To evaluate planar scaling with respect to $N$, we performed multiple runs in 
$[0,\,1]^2$ where $W = 2^{11}$ was fixed and $\mu$ and $\nu$ were uniform.
The $N = n$ points where $\nu = 1/n$ were placed at random locations in $A$. 
Because the resulting time data was highly dependent on point placement, it was 
extremely noisy.
Thus, we did ten runs for each $N$ and took the median.
We started with $N=128$, increasing by eights up to to $N=192$, for a 
total of 100 tests.
The results are shown in the right-hand columns of 
Table~\ref{tb:IndScaling}\subref{tb:Nscaling}.
See the right side of Table~\ref{tb:IndScaleEQ}
for scaling equations with respect to $N$.

\subsubsection{Scaling interaction of $W$ and 
$N$ on the plane}\label{s:bigN}
Increasing $N$ means one must consider the scaling behavior of the boundary 
method with respect to $N$, as described in \sref{s:Nscaling}, above.
However, there is another relevant limiting factor for large $N$: the 
decreasing area size $\mu(A_i) = n^{-1}$ runs up against the accuracy of the 
reconstruction. For the problem shown in \fref{f:bigN}, the area of each region 
is $5.0 \times 10^{-3}$.
When $w_* = 2^{-11}$, the maximum error for the 
area of the partition regions is $8.34 \times 10^{-4}$.
This is $16.7\%$, about one-sixth of the size of each 
region.

\begin{figure}[htpb]
  \centering
  \subfloat[$N=200$ points in $\R^2$]{%
    \label{f:Npoints}
    \centering
    \resizebox{0.3\textwidth}{!}{%
    \begin{overpic}[width=\textwidth]{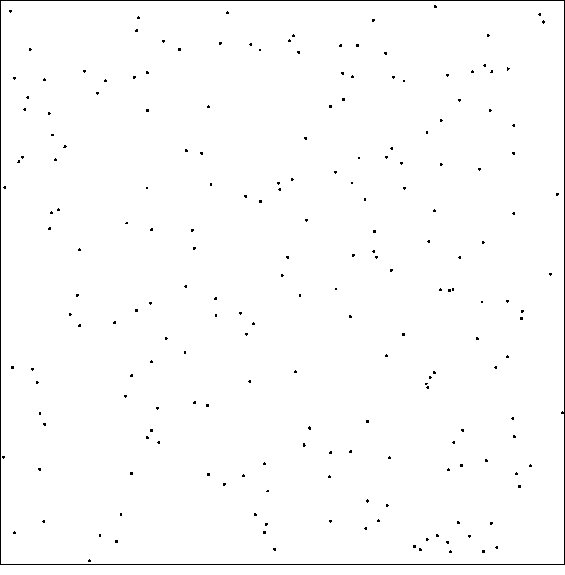}
    \end{overpic}}
  }%
  \quad\quad%
  \subfloat[$\mu$-Partition]{%
    \label{f:Npartitions}
    \centering
    \resizebox{0.3\textwidth}{!}{%
    \begin{overpic}[width=\textwidth]{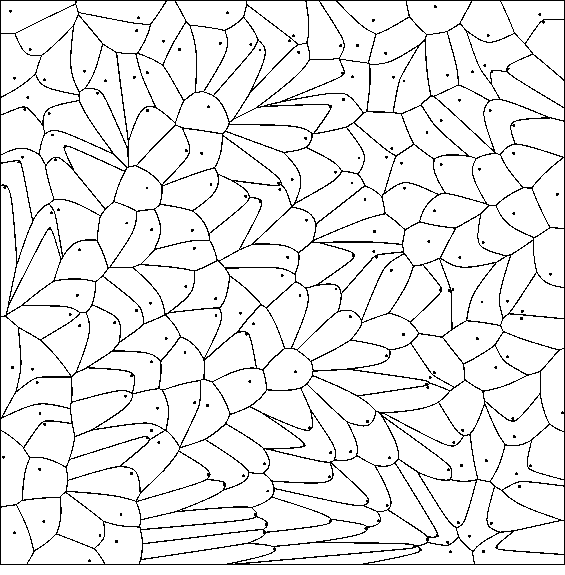}
    \end{overpic}}
  }%
\caption{Partitioning with large $N$}\label{f:bigN}
\end{figure}

If all we desire is the \Was{} distance or the boundary set, this error need not 
be a concern.
However, if we want an accurate set of shifts, 
large $N$ requires that we increase $W$ to match.
Hence, we wanted to consider what happens as $W$ and $N$ increase in 
tandem.

As it turns out, the scaling we observe is consistent with the product 
of the two scaling behaviors already determined: $\mathcal{O}(WN\log W\log N)$ 
with respect to time, and $\mathcal{O}(WN^{1/2})$ with respect to storage.
See Table~\ref{tb:WNscaleEQ} for approximate equations.

\begin{table}[htpb]
\centering
\caption{Time and memory scaling with respect to both $W$ and 
$N$}\label{tb:WNscaleEQ}%
\renewcommand{\arraystretch}{1.5}
\begin{tabular}{| c || l |}
\hline
Time
& $T(N,\,W) \approx 2.853\times10^{-6}WN\ln W \ln N$
\\\hline
Storage
& $S(N,\,W) \approx 1.538\times10^{-3}\, WN^{1/2}$
\\\hline
\end{tabular}
\end{table}

\subsubsection{Scaling in three dimensions and extrapolation to 
$\R^d$}\label{s:3dscale}
The computations described above can be repeated in three dimensions.
We scale $W$ separately by taking a projection of Example \ref{x:steps} into 
the center of the cube $[0,\,1]^3$. Then we consider the median of tests 
when $W = 2^{7}$ and $N$ ranges from 8 to 80. Finally, we scale $W$ and $N$ 
together, and consider their combined behavior. 
Approximate scaling equations are given in Table~\ref{tb:WN3DscaleEQ}.

\begin{table}[htpb]
\centering
\caption{3-D scaling with respect to $W$ and 
$N$}\label{tb:WN3DscaleEQ}%
\renewcommand{\arraystretch}{1.5}
\begin{tabular}{| c | c || l | }
\hline
\multirow{2}{*}{$W$ alone}
& Time
& $T(W) \approx 6.878\times10^{-5}W^2\ln W$
\\
& Storage
& $S(W) \approx 2.341\times10^{-2}\, W^2$
\\\hline\hline
\multirow{2}{*}{$N$ alone}
& Time
& $T(N) \approx 2.849\times10^{-1}N \ln N$
\\
& Storage
& $S(N) \approx 2.315\times10^{2}\, N^{1/3}$
\\\hline\hline
\multirow{2}{*}{$W$ and $N$}
& Time
& $T(N,\,W) \approx 3.531\times10^{-6}W^2N\ln W \ln N$
\\
& Storage
& $S(N,\,W) \approx 1.397\times10^{-1}\, W^2N^{1/3}$
\\\hline
\end{tabular}
\end{table}

Taking the combined scaling equations for two and three dimensions, and 
extrapolating to arbitrary dimension $d \geq 2$, we anticipate scaling 
of
\begin{equation*}
T(d,\,N,\,W) \sim \mathcal{O}(W^{d-1}N\log W\log N)
\quad\text{and}\quad
S(d,\,N,\,W) \sim \mathcal{O}(W^{d-1}N^{1/d}).
\end{equation*}

\section{Conclusions and future work}\label{s:develop}
In this work, we presented the boundary method, a new technique for 
approximating solutions to semi-discrete optimal transportation problems.  We gave an algorithmic description and mathematical justification.
As we showed, by tackling only the boundary of the regions to be 
transported, the method has very favorable scaling properties.
Under the assumption that all computations are exact, we gave sharp 
convergence results for $p$-norms with $p \in (1,\,\infty)$, and we
presented numerical
examples supporting those convergence results. 
We showed that the boundary method can provide accurate approximations of
the partition regions and \Was{} distance for a multitude of cost
functions, including some that are not covered by our theorems:
the $1$-norm, the $\infty$-norm, strictly convex non-norms such as the
squared $2$-norm, concave non-norms such as $p$-functions with
$p \in (0,\,1)$,
and polynomial combinations of $p$-functions that are
neither concave nor convex.
As we also showed, even when partitioning fails, the boundary method
can solve with accuracy and convergence comparable to the case where
a partition exists.
Our future work will consider applications of the boundary method to fully
continuous mass transportation problems and the impact of
estimated computations on convergence.

\bibliographystyle{elsarticle-num}
\bibliography{ref}
\end{document}